\documentclass[11pt]{amsart}
\usepackage{geometry}                
\usepackage{graphicx}
\usepackage[new]{old-arrows}
\usepackage{amssymb}
\usepackage{amsmath}
\usepackage{mathrsfs}
\usepackage{epstopdf}
\let\originallhook\lhook
\usepackage{MnSymbol}
\let\lhook\originallhook
\usepackage{stmaryrd}
\usepackage{verbatim}
\usepackage[all]{xy}
\usepackage[dvipsnames]{xcolor}
\title[Stable bundles on the fibres of the hyperk\"ahler twistor projection]{Families of stable bundles on the fibres of the hyperk\"ahler twistor projection}
\author{Artour Tomberg}
\thanks{The study has been funded within the framework of the HSE University Basic Research Program and the Russian Academic Excellence Project '5-100'.}
\address{Faculty of Mathematics, National Research University Higher School of Economics, 6 Usacheva St., Moscow, Russia, 119048}
\address{Department of Mathematics, Western University, Middlesex College, London, Ontario, Canada, N6A 5B7}
\email{artour@tomberg.com}

\DeclareMathOperator{\Hom}{Hom}

\DeclareMathOperator{\Imag}{Im}
\DeclareMathOperator{\id}{id}
\DeclareMathOperator{\Tw}{Tw}
\DeclareMathOperator{\Pic}{Pic}
\DeclareMathOperator{\rk}{rk}
\DeclareMathOperator{\Quot}{Quot}
\DeclareMathOperator{\Dou}{Dou}
\DeclareMathOperator{\Supp}{Supp}
\DeclareMathOperator{\Vol}{Vol}
\DeclareMathOperator{\Ker}{Ker}

\DeclareMathOperator{\Tor}{Tor}

\newcommand{\bigslant}[2]{{\raisebox{.2em}{$#1$}\left/\raisebox{-.2em}{$#2$}\right.}}
\newcommand{\hooklongrightarrow}{\lhook\joinrel\longrightarrow}

\numberwithin{equation}{section}

\begin{document}

\addtocontents{toc}{\protect\setcounter{tocdepth}{1}}

\begin{abstract}
Given a holomorphic vector bundle $E$ on the twistor space $\Tw(M)$ of a simple hyperk\"ahler manifold $M$, we view it as a family of bundles $\left\{E_I\right\}$ on the fibres $\pi^{-1}(I)$ of the twistor projection $\pi : \Tw(M) \to \mathbb{CP}^1$, and study the relationship between stability of $E$ and its fibrewise stability. We verify that the argument of Teleman establishing the Zariski openness of stability and semi-stability in families of bundles applies in the case of the family $\left\{E_I\right\}$. We prove a partial converse to a result of Kaledin and Verbitsky, showing that an irreducible bundle $E$ on $\Tw(M)$ is generically fibrewise stable if the rank of $E$ is 2 or 3, or at least one element of the family $\left\{E_I\right\}$ is a simple bundle, in the sense that $\Hom(E_I, E_I) = \mathbb{C}$.
\end{abstract}

\maketitle

\tableofcontents

\section{Introduction}

The interplay between stability of vector bundles and Hermitian-Einstein metrics has been a very active area of research in differential geometry in the 1980s, and the concerted efforts of many of the leading mathematicians of the time have led to the result that is now known as the Kobayashi-Hitchin correspondence (Theorem \ref{thm:kobayashi-hitchin}, see \cite{lubke-teleman} for a reference). Among its many other applications, this deep theorem gives a better understanding of the geometry of moduli spaces of stables bundles. These are mostly studied in the projective algebraic context, where there is a rich and developed theory (see for example \cite{huyb-lehn}). However, for non-algebraic, and especially for non-K\"ahler manifolds, the theory is more difficult. The first explicit description of the moduli space of stable bundles on a non-K\"ahler manifold was given by Braam and Hurtubise in the paper \cite{braam-hurtubise}, where they studied stable $SL(2, \mathbb{C})$-bundles on primary elliptic Hopf surfaces. Since then, while there have been further results about moduli spaces and their structure (see \cite{brin-mor,apr-mor-toma,brin-halan-traut}), much remains unknown in the non-K\"ahler case.

One of the things that makes the non-K\"ahler case more difficult is the presence of irreducible bundles. These are bundles which don't have any proper subsheaves of lower rank, and are thus always stable, vacuously. They don't occur on algebraic manifolds, where one always has recourse to various filtrations. In contrast to filtrable bundles, irreducible bundles cannot be constructed as extensions of coherent sheaves, and no general method of constructing such bundles is known.

The subject of this paper is stability of vector bundles on a compact hyperk\"ahler manifold $M$ and its twistor space $\Tw(M)$. Recall that a hyperk\"ahler manifold $M$ comes equipped with a (non-K\"ahler) twistor space $\Tw(M)$ and a holomorphic twistor projection $\pi : \Tw(M) \to \mathbb{CP}^1$, whose fibres $M_I := \pi^{-1}(I)$, $I \in \mathbb{CP}^1$, although K\"ahler, are mostly non-algebraic (see Proposition \ref{thm:dense}). We will view a holomorphic vector bundle $E$ on $\Tw(M)$ as a family $\left\{E_I := \left.E\right|_{M_I}\right\}$ of bundles on the fibres $M_I$ of the projection $\pi$. A natural question to consider in this context is how the stability of $E$ as a bundle on $\Tw(M)$ is related to the stability of its restrictions $E_I$.

In the paper \cite{kaled-verbit}, Kaledin and Verbitsky prove, among other things, that if $E$ is generically fibrewise stable, that is, its restriction $E_I$ is stable for Zariski generic $I \in \mathbb{CP}^1$, then $E$ itself is stable as a bundle on $\Tw(M)$ (Lemma 7.3 in \cite{kaled-verbit}). The converse to this statement does not hold, as was shown in \cite{tomberg2}, which gives an explicit example of a stable bundle $E$ on $\Tw(M)$ which is nowhere fibrewise stable. However, as is evident from the proof of Lemma 7.3 in \cite{kaled-verbit}, the statement of the result of Kaledin and Verbitsky can be made stronger in the following way.

\theoremstyle{plain}
\newtheorem*{forward}{Theorem}
\begin{forward} \label{thm:forward}
\emph{\textbf{(Kaledin-Verbitsky)}} Let $M$ be a compact hyperk\"ahler manifold and let $E$ be a holomorphic vector bundle on the twistor space $\Tw(M)$. If $E$ is generically fibrewise stable, then it is irreducible.
\end{forward}

For a proof, see the first part of the proof of Theorem \ref{thm:result}. One may ask whether the converse to this stronger version of the result is in fact true. In other words, given an irreducible bundle $E$ on the twistor space $\Tw(M)$, will it always be stable on the generic fibre of the twistor projection $\pi : \Tw(M) \to \mathbb{CP}^1$? If the answer is yes, this would give a very nice characterization of irreducible bundles on the twistor space $\Tw(M)$ of a hyperk\"ahler manifold $M$. While we don't provide a full answer to this question in this paper, we will prove a partial converse to the above theorem. Namely, we show that for a simple hyperk\"ahler manifold $M$, an irreducible bundle $E$ on the twistor space $\Tw(M)$ is generically fibrewise stable provided that $E$ has rank 2 or 3, or at least one of the $E_I$ in the corresponding family is a simple holomorphic bundle, in the sense that $\Hom(E_I, E_I) = \mathbb{C}$.

The present article is a somewhat abridged version of the second part of the author's PhD thesis \cite{dissertaciya}. It is organized as follows. Section \ref{sect:preliminaries} is a review of hyperk\"ahler geometry and the theory of stability, where we state some technical results used in subsequent sections. In Section \ref{sect:teleman} we prove that for a family $\left\{E_I\right\}$ of bundles on the fibres of the twistor projection $\pi : \Tw(M) \to \mathbb{CP}^1$, fibrewise stability and semi-stability are Zariski open conditions on the base $\mathbb{CP}^1$. The proof is basically a verification that the argument of Teleman from \cite{teleman} adapts to the case of the twistor projection. Section \ref{sect:main} contains the proof of the main result of the article, namely that the converse to the above theorem holds in case $M$ is simple hyperk\"ahler and $E$ is of rank 2 or 3, or of general rank with the extra assumption of being simple on at least one fibre of the projection $\pi : \Tw(M) \to \mathbb{CP}^1$.

The author would like to thank Misha Verbitsky and Jacques Hurtubise for their help and support in the preparation of this manuscript.

\section{Preliminaries} \label{sect:preliminaries}

We start by giving definitions of the objects that we will be working with and stating results (mostly without proof) that will be useful for us in subsequent sections. The main references will be the papers \cite{verbit1}, \cite{verbit2}, \cite{kaled-verbit} for hyperk\"ahler geometry, and the books \cite{kobayashi}, \cite{lubke-teleman}, \cite{okonek} for stability of vector bundles.

\theoremstyle{definition}
\newtheorem{hyperk}{Definition}[section]
\begin{hyperk}
A \emph{hyperk\"ahler} manifold is a smooth manifold $M$ together with a triple of integrable almost complex structures $I, J, K : TM \to TM$ that satisfy the quaternionic relations $I^2 = J^2 = K^2 = -1$, $IJ = -JI = K$, and a Riemannian metric $g$, which is simultaneously Hermitian with respect to the structures $I, J, K$, and whose Levi-Civita connection $\nabla$ satisfies $\nabla I = \nabla J = \nabla K = 0$.
\end{hyperk}

Note that $\nabla$ also preserves the corresponding Hermitian forms $\omega_I, \omega_J, \omega_K$ of $g$, so that $g$ is simultaneously K\"ahler with respect to these structures. In fact, there are many more K\"ahler structures on $M$.

Together with the identity mapping, $I, J, K$ induce an action of the quaternion algebra $\mathbb{H}$ on the tangent bundle $TM$, which is moreover parallel with respect to $\nabla$. A straightforward verification shows that any linear combination $A = aI + bJ + cK$ with $a^2 + b^2 + c^2 = 1$ satisfies $A^2 = -1$ and $\nabla A = 0$, and is thus an integrable almost complex structure on $M$, for which $g$ is again a K\"ahler metric. Thus, a hyperk\"ahler manifold has a whole 2-sphere of \emph{induced complex structures}
\[
S^2  = \left\{aI + bJ + cK : a^2 + b^2 + c^2 = 1 \right\} = \left\{A \in \mathbb{H} : A^2 = -1\right\} \subseteq \Imag \mathbb{H},
\]
which we would like to encode as a single geometrical object.

\theoremstyle{definition}
\newtheorem{twistor}[hyperk]{Definition}
\begin{twistor}
Let $M$ be hyperk\"ahler. The product manifold $\Tw(M) = M \times S^2$ is called the \emph{twistor space} of $M$.
\end{twistor}

Thinking of $S^2$ as the set of induced complex structures of $M$ as above, the twistor space parametrizes these structures at points of $M$. Identifying $S^2$ with $\mathbb{CP}^1$ in the usual way, we can give $\Tw(M) \cong M \times \mathbb{CP}^1$ a natural complex structure: for any point $(m, A) \in M \times \mathbb{CP}^1$ we define $\mathcal{I} :  T_{(m,A)}\Tw(M) \to T_{(m,A)}\Tw(M)$ by
\[
\begin{array}{ccccc}
\mathcal{I} & : & T_m M \oplus T_A \mathbb{CP}^1 & \longrightarrow & T_m M \oplus T_A \mathbb{CP}^1, \\
&& (X, V) & \longmapsto & \left(AX, I_{\mathbb{CP}^1}V\right)
\end{array}
\]
where $I_{\mathbb{CP}^1} : T\mathbb{CP}^1 \to T\mathbb{CP}^1$ is the usual complex structure on $\mathbb{CP}^1$. It's easy to check that this defines an almost complex structure on $\Tw(M)$, which is in fact integrable \cite{sal}, thus making $\Tw(M)$ into a complex manifold. There are canonical projections
\[
\xymatrix{& \Tw(M) \ar[rd]^\pi \ar[dl]_\sigma & \\ M & & \mathbb{CP}^1,}
\]
the second of which is a holomorphic map. The hyperk\"ahler metric $g$ on $M$ and the Fubini-Study metric $g_{\mathbb{CP}^1}$ on $\mathbb{CP}^1$ induce the Hermitian metric
\[
\sigma^*(g) + \pi^*\left(g_{\mathbb{CP}^1}\right)
\]
on $\Tw(M)$, which, although never K\"ahler \cite{hitchin}, satisfies the weaker property of being \emph{balanced} (see \cite{kaled-verbit}), i.e. its Hermitian form $\omega$ satisfies $d\left(\omega^{n-1}\right) = 0$, where $\dim_{\mathbb{C}} \Tw(M) = n$.

When considering the totality of the induced complex structures on $M$, the particular structures $I, J, K$ no longer play any vital role, so from now on we will denote an arbitrary induced complex structure by $I$, the resulting K\"ahler manifold by $M_I$ and its structure sheaf by $\mathcal{O}_I$. The $M_I$ are precisely the fibres of the holomorphic twistor projection $\pi : \mathrm{Tw}(M) \to \mathbb{CP}^1$, and it will be useful to think of $\Tw(M)$ as the collection of K\"ahler manifolds $M_I$ lying above the points $I \in \mathbb{CP}^1$ via the map $\pi$. A compact hyperk\"ahler manifold $M$ is called \emph{simple} if it is simply connected and satisfies $H^{2,0}(M_I) = \mathbb{C}$ for some (and hence for all) $I \in \mathbb{CP}^1$. In what follows, all our hyperk\"ahler manifolds will be assumed to be compact and simple. 

Recall that a hyperk\"ahler manifold $M$ is equipped with a parallel action of the quaternion algebra $\mathbb{H}$ on its tangent bundle. Restricting to the group of unitary quaternions in $\mathbb{H}$, we get an action of $SU(2)$ on $TM$, hence on all of its tensor bundles, and in particular on the bundle of differential forms $\Lambda^*M$. Since the action is parallel, it commutes with the Laplace operator, and thus preserves harmonic forms. Applying Hodge theory, we get a natural action of $SU(2)$ on the cohomology $H^*(M,\mathbb{C})$.

\theoremstyle{plain}
\newtheorem{invariant}[hyperk]{Lemma}
\begin{invariant} \label{thm:invariant}
A differential form $\eta$ on a hyperk\"ahler manifold $M$ is $SU(2)$-invariant if and only if it is of Hodge type $(p,p)$ with respect to all induced complex structures $M_I$.
\end{invariant}

\begin{proof}
Proposition 1.2 in \cite{verbit1}.
\end{proof}

\theoremstyle{definition}
\newtheorem{generic}[hyperk]{Definition}
\begin{generic}
Let $M$ be hyperk\"ahler and $I$ an induced complex structure. We say that $I$ is \emph{generic} with respect to the hyperk\"ahler structure on $M$ if all elements in
\[
\bigoplus_p H^{p,p}(M_I) \cap H^{2p}(M,\mathbb{Z}) \subset H^*(M,\mathbb{C})
\]
are $SU(2)$-invariant.
\end{generic}

This terminology is justified: most induced complex structures are generic, in a sense made precise in the following proposition.

\theoremstyle{plain}
\newtheorem{dense}[hyperk]{Proposition}
\begin{dense} \label{thm:dense}
Let $M$ be a hyperk\"ahler manifold. The set $S_0 \subset S^2$ of generic induced complex structures is dense in $S^2$ and its complement is countable.
\end{dense}

\begin{proof}
Proposition 2.2 in \cite{verbit2}.
\end{proof}

As we will see, the genericity of the complex structure $I$ puts rigid conditions on the geometric structure of the manifold $M_I$. For instance, all line bundles on $M_I$ have only zero or nowhere vanishing sections (see the proof of Proposition \ref{thm:divisors}), hence $M_I$ can never be algebraic since it has no effective divisors.

Our next goal is to give the definition of stable vector bundles (and, more generally, torsion-free coherent sheaves), as well as Hermitian-Einstein metrics and the Kobayashi-Hitchin correspondence. In what follows, $M$ will be a compact complex manifold of dimension $n$ and $g$ a Hermitian metric on $M$ with Hermitian form $\omega$. We will denote by $\mathcal{O}$ the sheaf of holomorphic functions on $M$.

Recall that for a coherent sheaf $\mathcal{F}$ over $M$, we have the \emph{dual sheaf} $\mathcal{F}^* = \mathcal{H}om\left(\mathcal{F}, \mathcal{O}\right)$, as well as a natural morphism into the double dual $\mathcal{F} \longrightarrow \mathcal{F}^{**}$, whose kernel consists precisely of the torsion elements of $\mathcal{F}$.

\theoremstyle{definition}
\newtheorem{torsionfree}[hyperk]{Definition}
\begin{torsionfree}
A coherent sheaf $\mathcal{F}$ over $M$ is called \emph{torsion-free} if $\forall x \in M$, the stalk $\mathcal{F}_x$ is a torsion-free $\mathcal{O}_x$-module, or equivalently, if the natural morphism of sheaves
\[
\mathcal{F} \longrightarrow \mathcal{F}^{**}
\]
is injective. If it is an isomorphism, we say that $\mathcal{F}$ is \emph{reflexive}. We call the sheaf $\mathcal{F}$ \emph{normal} if for every open set $U \subseteq M$ and every analytic subset $A \subset U$ of codimension at least 2, the restriction map
\[
\mathcal{F}(U) \longrightarrow \mathcal{F}(U \setminus A)
\]
is an isomorphism.
\end{torsionfree}

Clearly, a vector bundle $E$, viewed as a locally free sheaf, is reflexive (and hence torsion-free). On the other hand, for an arbitrary coherent sheaf $\mathcal{F}$, let
\[
S(\mathcal{F}) = \left\{x \in M : \mathcal{F}_x \textrm{ is not free over } \mathcal{O}_x\right\}
\]
denote the \emph{singularity set} of $\mathcal{F}$. It can be shown (see Section $\S1$ of Chapter 2 in \cite{okonek}) that for an arbitrary coherent sheaf this is a closed analytic subset of $M$ of codimension $\ge 1$ ($\ge 2$ for a torsion-free sheaf, $\ge 3$ for a reflexive sheaf), so that $\mathcal{F}$ restricted to $M \setminus S(\mathcal{F})$ is locally free. This justifies the following definition.

\theoremstyle{definition}
\newtheorem{rank}[hyperk]{Definition}
\begin{rank}
The \emph{rank} $\rk \mathcal{F}$ of a coherent sheaf $\mathcal{F}$ over $M$ is the rank of the locally free sheaf
\[
\left.\mathcal{F}\right|_{M \setminus S(\mathcal{F})} \textrm{ over } M \setminus S(\mathcal{F}).
\]
\end{rank}

Recall that for an arbitrary coherent sheaf $\mathcal{F}$ and any integer $s \ge 0$, we can define the exterior power sheaf $\Lambda^s \mathcal{F}$. If $\mathcal{F}$ is a torsion-free sheaf of rank $s > 0$, then $\Lambda^s \mathcal{F}$ has rank 1, and the \emph{determinant} of $\mathcal{F}$,
\[
\det \mathcal{F} := \left(\Lambda^s \mathcal{F}\right)^{**},
\]
is actually a line bundle on $M$, since the dual of an arbitrary coherent sheaf is reflexive (Proposition V.5.18 in \cite{kobayashi}), and a reflexive sheaf of rank 1 is a line bundle (Lemma 1.1.15 in Chapter 2 of \cite{okonek}).

To proceed with the definition of degree, we need to impose a certain differential condition on the metic $g$.

\theoremstyle{definition}
\newtheorem{gauduchon}[hyperk]{Definition}
\begin{gauduchon}
The metric $g$ is called \emph{Gauduchon} if it satisfies the condition
\[
\partial \bar{\partial} \left(\omega^{n-1}\right) = 0.
\]
\end{gauduchon}

\theoremstyle{definition}
\newtheorem{degree}[hyperk]{Definition}
\begin{degree}
Let $g$ be Gauduchon. The \emph{degree} of a torsion-free coherent sheaf $\mathcal{F}$ on $M$ with respect to $g$ is given by
\[
\deg_g (\mathcal{F}) =\int_M c_1(\det \mathcal{F}, h) \wedge \omega^{n-1},
\]
where $h$ is an arbitrary Hermitian metric on the line bundle $\det \mathcal{F}$, and
\[
c_1(\det \mathcal{F}, h) := \frac{\sqrt{-1}}{2\pi} R^h,
\]
where $R^h \in \Gamma(\Lambda^{1,1}M)$ is the curvature form of the Chern connection on $\left(\det\mathcal{F}, h\right)$. We will write $\deg(\mathcal{F})$ when the metric will be clear from the context.
\end{degree}

The Gauduchon condition on $g$ ensures that $\deg_g(\mathcal{F})$ is well-defined and does not depend on the metric $h$ (see Lemma 1.1.18 in \cite{lubke-teleman}). If the metric $g$ satisfies the balancedness condition $d\left(\omega^{n-1}\right) = 0$ (in particular, if it is K\"ahler), then it is clearly Gauduchon, and in fact the degree only depends on the first Chern class $c_1(\det \mathcal{F})$, making it a topological invariant of $\det \mathcal{F}$; for an arbitrary Gauduchon metric it is only a holomorphic invariant of $\det \mathcal{F}$. Note that, for an arbitrary Hermitian metric $g$ on $M$, the definition of degree does not make sense, however, as shown in the following theorem proved in \cite{gauduchon}, the conformal class of $g$ always contains a Gauduchon metric, which is essentially unique.

\theoremstyle{plain}
\newtheorem{gaud}[hyperk]{Theorem}
\begin{gaud}
If $M$ is compact, then for every Hermitian metric $g$ on $M$ there exists a positive function $\varphi \in C^{\infty}(M, \mathbb{R}^{>0})$ such that
\[
g' := \varphi \cdot g
\]
is Gauduchon. If $M$ is connected and $n \ge 2$, then $g'$ is unique up to a positive constant.
\end{gaud}

We are now ready to define stability of torsion-free coherent sheaves on $M$.

\theoremstyle{definition}
\newtheorem{stability}[hyperk]{Definition}
\begin{stability}
Let $g$ be a Gauduchon metric on $M$, and let $\mathcal{F}$ be a nontrivial torsion-free coherent sheaf. The $g$-\emph{slope} of $\mathcal{F}$ is given by
\[
\mu_g(\mathcal{F}) := \frac{\deg_g(\mathcal{F})}{\rk(\mathcal{F})},
\]
denoted simply by $\mu(\mathcal{F})$ when the metric is clear from the context. The sheaf $\mathcal{F}$ is called $g$-\emph{stable} (resp. $g$-\emph{semi-stable}) if for every subsheaf $\mathcal{G} \subset \mathcal{F}$ with $0 < \rk(\mathcal{G}) < \rk(\mathcal{F})$ we have
\[
\mu_g(\mathcal{G}) < \mu_g(\mathcal{F}) \ \left(\textrm{resp. } \mu_g(\mathcal{G}) \le \mu_g(\mathcal{F})\right).
\]
$\mathcal{F}$ is called $g$-\emph{polystable} if it is a direct sum of $g$-stable bundles of the same slope. It is called \emph{irreducible} if it has no proper subsheaves of lower rank.
\end{stability}

It's clear that an irreducible $\mathcal{F}$ is stable with respect to any metric on $M$. It's also clear that for a reflexive sheaf, $\mathcal{F}$ is irreducible if and only $\mathcal{F}^*$ is irreducible.

We now give the definition of Hermitian-Einstein structures on a holomorphic vector bundle $E$ over $M$, a concept which is intimately related to the notion of stability, in a sense that will be made precise later. Recall that the Hermitian structure $g$ on $M$ defines a linear operator on the bundle of differential forms of $M$ given by exterior multiplication with the Hermitian form $\omega$:
\[
\begin{array}{ccccc}
L_g & : & \Lambda^{p,q} M & \longrightarrow & \Lambda^{p+1,q+1}M \\
&& \alpha & \longmapsto & \alpha \wedge \omega
\end{array}
\]
We will denote the $g$-adjoint operator of $L_g$ by $\Lambda_g : \Lambda^{p,q}M \to \Lambda^{p-1,q-1}M$. It can be shown that for a $(1,1)$-form $\alpha$, $\Lambda_g(\alpha)$ satisfies the following identity:
\[
\alpha \wedge \omega^{n-1} = \frac{1}{n} \, \Lambda_g(\alpha) \omega^n.
\]

\newtheorem{hermit-einstein}[hyperk]{Definition}
\begin{hermit-einstein}
A Hermitian metric $h$ on a holomorphic vector bundle $E$ on $M$ is called $g$-\emph{Hermitian-Einstein} if the curvature $R^h \in \Gamma(\Lambda^{1,1}M \otimes E^* \otimes E)$ of its Chern connection satisfies the equation
\[
\sqrt{-1}\Lambda_g R^h = \gamma \cdot \id_E,
\]
or equivalently
\[
\left(\sqrt{-1}R^h\right) \wedge \omega^{n-1} = \frac{\gamma}{n} \omega^n \cdot \id_E,
\]
where $\gamma$ is a real constant, called the \emph{Einstein constant} of $h$.
\end{hermit-einstein}

If the holomorphic vector bundle $E$ is \emph{simple}, in the sense that $\Hom(E, E) = \mathbb{C}$, a $g$-Hermitian-Einstein metric on $E$ (if it exists) is unique up to a positive scalar (Proposition 2.2.2 in \cite{lubke-teleman}). In case the metric $g$ is Gauduchon, the Einstein constant is proportional to the degree of the vector bundle $E$, as the following proposition shows.

\theoremstyle{plain}
\newtheorem{einst-slope}[hyperk]{Proposition}
\begin{einst-slope} \label{thm:einst-slope}
If $g$ is Gauduchon and $h$ a Hermitian-Einstein metric on $E$ with Einstein constant $\gamma$, then
\[
\gamma = \frac{2\pi}{(n-1)! \cdot \Vol_g(M)} \cdot \mu_g(E),
\]
where
\[
\Vol_g(M) = \int_M \frac{1}{n!} \cdot \omega^n
\]
is the volume of $M$ with respect to the metric $g$.
\end{einst-slope}

\begin{proof}
Lemma 2.1.8 in \cite{lubke-teleman}.
\end{proof}

There is an intimate relationship between Hermitian-Einstein structures and stability. The following fundamental theorem, whose proof is the subject of the book \cite{lubke-teleman}, shows that the two notions are essentially equivalent.

\theoremstyle{plain}
\newtheorem{kobayashi-hitchin}[hyperk]{Theorem}
\begin{kobayashi-hitchin} \label{thm:kobayashi-hitchin}
If $g$ is a Gauduchon metric and $E$ is a holomorphic vector bundle on $M$, then $E$ admits a Hermitian-Einstein metric if and only if it is polystable. In case the bundle is stable, this metric is unique up to a positive constant.
\end{kobayashi-hitchin}

This result is called the \emph{Kobayashi-Hitchin correspondence}, after the people who conjectured it, and was proved in increasing generality by various mathematicians, among whom the greatest contributions were by Donaldson \cite{donaldson, donaldson2, donaldson3}, Uhlenbeck and Yau \cite{uhlenbeck-yau, uhlenbeck-yau2}, Buchdahl \cite{buchdahl} and Li and Yau \cite{li-yau}.

As a first application of this theorem, we see that a holomorphic line bundle $L$ on a Gauduchon manifold $M$, which is clearly stable in any metric, always admits a Hermitian-Einstein metric. As a matter of fact, it can be shown (Corollary 2.1.6 in \cite{lubke-teleman}) that for an arbitrary Hermitian metric $g$ on $M$, a holomorphic line bundle $L$ on $M$ always admits a $g$-Hermitian-Einstein metric $h$, unique up to constant rescaling. Thus, we can define the Einstein constant of $L$ with respect to the Hermitian metric $g$ by
\[
\gamma_g(L) := \sqrt{-1}\Lambda_g R^h,
\]
where $R^h$ is the curvature of the Chern connection of $(L, h)$. Since the Chern connection stays the same when the metric is multiplied by a constant, $\gamma_g(L)$ is well-defined. We then have the following result.

\theoremstyle{plain}
\newtheorem{degree-proportional}[hyperk]{Proposition}
\begin{degree-proportional} \label{thm:degree-proportional}
Let $g'$ be a Gauduchon metric in the conformal class of $g$. Then there exists a positive constant $c$, depending only on $g$ and $g'$, such that
\[
\deg_{g'}(L) = c \cdot \gamma_g(L)
\]
for all holomorphic line bundles $L$ on $M$.
\end{degree-proportional}

\begin{proof} Proposition 1.3.16 in \cite{lubke-teleman}.
\end{proof}

Now let $M$ be again a compact hyperk\"ahler manifold with hyperk\"ahler metric $g$, and let $S_0 \subset S^2 \cong \mathbb{CP}^1$ denote the set of generic complex structures of $M$, as in the statement of Proposition \ref{thm:dense}. We have the following lemma.

\theoremstyle{plain}
\newtheorem{degr}[hyperk]{Lemma}
\begin{degr} \label{thm:degr}
An $SU(2)$-invariant 2-form $\beta$ on a hyperk\"ahler manifold $M$ satisifes
\[
\Lambda_I \beta = 0
\]
for any induced complex structure $I$, where by $\Lambda_I$ we mean the operator $\Lambda_g$ on the manifold $M_I$.
\end{degr}

\begin{proof}
Lemma 2.1 in in \cite{verbit1}.
\end{proof}

\theoremstyle{plain}
\newtheorem{divisors}[hyperk]{Proposition}
\begin{divisors} \label{thm:divisors}
Let $M$ be a hyperk\"ahler manifold and $\Tw(M)$ its twistor space. The holomorphic twistor projection $\pi : \Tw(M) \to \mathbb{CP}^1$ establishes a one-to-one correspondence between divisors on $\mathbb{CP}^1$ and those on $\Tw(M)$.
\end{divisors}

\begin{proof}
It suffices to show that the only (irreducible) hypersurfaces on $\Tw(M)$ are the fibres of the twistor projection $\pi : \Tw(M) \to \mathbb{CP}^1$. Suppose this is not so, and $V \subset \Tw(M)$ is an irreducible hypersurface which is not a fibre of $\pi$. Then $\pi(V) = \mathbb{CP}^1$, so that $V$ intersects every fibre of $\pi$. We can choose a generic structure $I \in S_0$ so that the restriction $V \cap \pi^{-1}(I) = V \cap M_I$ is a divisor on $M_I$. Letting $L$ be the line bundle corresponding to this divisor, the first Chern class $c_1(L) \in H^{1,1}(M_I) \cap H^2(M, \mathbb{Z})$ is $SU(2)$-invariant by genericity of $I$. Letting $\eta$ denote the harmonic form representing $c_1(L)$, it's clear that $\eta$ is $SU(2)$-invariant as a differential form. By Proposition II.2.23 in \cite{kobayashi}, there is a Hermitian metric $h$ on $L$ such that  $c_1(L, h) = \eta$; in other words,
\[
\frac{\sqrt{-1}}{2\pi}R^h = \eta,
\]
where $R^h$ is the curvature of the Chern connection of $(L, h)$. Since $R^h$ is $SU(2)$-invariant, it follows from Lemma \ref{thm:degr} that $(L, h)$ is Hermitian-Einstein with Einstein constant 0, so by Proposition \ref{thm:einst-slope} (or Proposition \ref{thm:degree-proportional}), $\deg_g(L) = 0$. But then, as a consequence of the Poincar\'e-Lelong formula (see Proposition 1.3.5 in \cite{lubke-teleman}), $L$ is either trivial or $H^0(M_I, L) = 0$, which contradicts the construction of $L$ as the line bundle of an effective divisor on $M_I$. In fact, the argument shows that $M_I$ has no effective divisors and thus cannot be algebraic. In particular, $V \subseteq \Tw(M)$ as chosen above cannot exist.
\end{proof}

Just as shown in the proof of the above proposition for line bundles, it follows from Lemma \ref{thm:degr} that for a generic complex structure $I \in S_0$, any holomorphic vector bundle over $M_I$, and in fact any torsion-free sheaf, has degree 0, and in particular, is semi-stable.

We close this section with one further notion of stability which is defined for a hyperk\"ahler manifold $M$. Recall that $M$ comes equipped with a twistor space $\Tw(M)$ and a holomorphic projection $\pi : \Tw(M) \to \mathbb{CP}^1$. As mentioned previously, $\Tw(M)$ has a natural balanced metric, thus it makes sense to talk about stable vector bundles and torsion-free sheaves on $\Tw(M)$. In fact, because $\Tw(M)$ parametrizes the totality of the complex structures on $M$ induced by the hyperk\"ahler structure, it makes sense to make the following definition.

\theoremstyle{definition}
\newtheorem{fibrewise}[hyperk]{Definition}
\begin{fibrewise}
A holomorphic vector bundle $E$ on $\Tw(M)$ is called \emph{fibrewise stable} if its restriction $E_I$ to $M_I$ is stable in the induced K\"ahler metric for each $I \in \mathbb{CP}^1$. $E$ is called \emph{generically fibrewise stable} if $E_I$ is stable for all $I$ in a nonempty Zariski open subset of $\mathbb{CP}^1$.
\end{fibrewise}

\section{Zariski openness of fibrewise stability for the twistor projection} \label{sect:teleman}

It is a general fact, which can be made precise, that given a morphism of spaces $f : X \to S$ and a vector bundle $E$ on $X$, thought of as a family of vector bundles $\left\{E_s : s \in S\right\}$ on the fibres $\left\{f^{-1}(s) : s \in S\right\}$, the set
\[
S^{\mathrm{st}} := \left\{ s \in S : E_s \textrm{ is stable} \right\}
\]
is open in $S$ under some assumptions on the morphism $f : X \to S$. This holds true in the projective algebraic setting (see Proposition 2.3.1 in \cite{huyb-lehn}), where the topology on $S$ is understood to be the Zariski topology, as well as in the complex hermitian setting (see \cite{lubke-teleman}, Theorem 5.1.1), where the topology on $S$ is the usual Euclidean manifold topology. We would like to study families of vector bundles on the fibres of the twistor projection $\pi : \Tw(M) \to \mathbb{CP}^1$ of a simple hyperk\"ahler manifold $M$, and while it is natural to work in the Zariski topology on $\mathbb{CP}^1$, the twistor space $\Tw(M)$ is never projective (not even K\"ahler). We could apply the result of \cite{lubke-teleman} and conclude that the stability condition is open in $\mathbb{CP}^1$ in the classical topology, but for our purposes we would like to have the stronger result of Zariski openness.

In the paper \cite{teleman}, Teleman proves, among other results, the following theorem.

\theoremstyle{plain}
\newtheorem{stable-teleman}{Theorem}[section]
\begin{stable-teleman}
Let $Y$ be a compact connected Gauduchon manifold, $S$ an arbitrary complex manifold, and $E$ a holomorphic vector bundle on $Y \times S$, thought of as a family of vector bundles on $Y$ parametrized by the projection $Y \times S \to S$. Then the sets
\[
S^{\mathrm{st}} = \left\{ s \in S : E_s \textrm{ is stable} \right\}, \, S^{\mathrm{sst}} = \left\{ s \in S : E_s \textrm{ is semi-stable} \right\}
\]
are Zariski open provided the parameter manifold $S$ is compact.
\end{stable-teleman}

Although the twistor space $\Tw(M)$ is topologically a product $M \times S^2$, we cannot apply this theorem directly since it is not a complex analytic product of $M$ and $\mathbb{CP}^1$. We would thus like to extend Teleman's result in the slightly more general setting of the complex structure $M_I$ varying on the fibres of the twistor projection $\pi : \Tw(M) \to \mathbb{CP}^1$.

\theoremstyle{plain}
\newtheorem{stable-twistor}[stable-teleman]{Theorem}
\begin{stable-twistor} \label{thm:stable-twistor}
Let $M$ be a compact simple hyperk\"ahler manifold with twistor space $\Tw(M)$, and $E$ a holomorphic vector bundle on $\Tw(M)$ of rank $r$. Then the sets
\[
\left(\mathbb{CP}^1\right)^{\mathrm{st}} = \left\{ I \in \mathbb{CP}^1 : E_I \textrm{ is stable} \right\}, \, \left(\mathbb{CP}^1\right)^{\mathrm{sst}} = \left\{ I \in \mathbb{CP}^1 : E_I \textrm{ is semi-stable} \right\}
\]
are Zariski open in $\mathbb{CP}^1$.
\end{stable-twistor}

The proof will essentially be a verification that Teleman's argument in \cite{teleman} works for the twistor projection $\pi : \Tw(M) \to \mathbb{CP}^1$. We will work in the category of complex analytic spaces and their morphisms (for the definitions, see, for instance, \cite{grauert-remmert}). We start by studying the \emph{relative Picard group} of the twistor projection $\pi : \Tw(M) \to \mathbb{CP}^1$, which will be an object $\Pic_{\mathbb{CP}^1} \Tw(M)$ parametrizing the Picard groups $\left\{ \Pic M_I : I \in \mathbb{CP}^1 \right\}$ of the fibres of $\pi$. Generally, the existence of the relative Picard group is a delicate matter (see \cite{kleiman} for results in the algebraic setting). However, in our situation, $\Pic_{\mathbb{CP}^1} \Tw(M)$ has an explicit description, which we now give.

Set-theoretically,
\[
\Pic_{\mathbb{CP}^1} \Tw(M) := \bigslant{\left\{ \left(I, L\right) : I \in \mathbb{CP}^1, L \textrm{ holomorphic line bundle on } M_I \right\}}{\sim},
\]
where $\sim$ identifies isomorphic line bundles. Clearly, the fibres of the natural projection $\Pic_{\mathbb{CP}^1} \Tw(M) \to \mathbb{CP}^1$ are just the Picard groups $\Pic M_I$. Now fix an induced complex structure $I \in \mathbb{CP}$ on $M$. The exponential sheaf sequence
\[
0 \longrightarrow \mathbb{Z} \longrightarrow \mathcal{O}_{I} \stackrel{\exp}{\longrightarrow} \mathcal{O}^*_{I} \longrightarrow 0
\]
gives rise to a long exact sequence in cohomology, a portion of which looks like
\[
0 \longrightarrow H^1(M, \mathbb{Z}) \longrightarrow H^1(M_I, \mathcal{O}_I) \longrightarrow \Pic M_I \stackrel{c_1}{\longrightarrow} H^2 (M, \mathbb{Z})
\]
Since $M$ is simply connected, $H^1(M, \mathbb{Z}) = 0$ and $H^2(M, \mathbb{Z})$ has no torsion. By Hodge theory, $H^1(M_I, \mathcal{O}_I) = 0$, and by the Lefschetz theorem on $(1,1)$-classes, the image of $\Pic M_I \stackrel{c_1}{\to} H^2 (M, \mathbb{Z})$ is equal to $H^{1,1}(M_I) \cap H^2(M, \mathbb{Z})$. It follows from all this that $\Pic M_I$ is discrete and can be identified with the free abelian subgroup $H^{1,1}(M_I) \cap H^2(M, \mathbb{Z}) \subseteq H^2(M, \mathbb{Z})$. Hence we have a set-theoretic embedding 
\begin{equation} \label{vlozhenie}
\begin{array}{ccc}
\Pic_{\mathbb{CP}^1} \Tw(M) & \longhookrightarrow & \mathbb{CP}^1 \times H^2(M, \mathbb{Z}) \\
(I, L) & \longmapsto & \left(I, c_1(L)\right)
\end{array}
\end{equation}
and we would like to show that its image is a closed analytic subset of $\mathbb{CP}^1 \times H^2(M, \mathbb{Z})$ which we think of as the disjoint union of countably many copies of $\mathbb{CP}^1$ indexed by $H^2(M, \mathbb{Z})$.

To see this, let $L$ be a holomorphic line bundle on $M_{I_0}$ for some $I_0 \in \mathbb{CP}^1$. We think of $L$ as an element of $\Pic_{\mathbb{CP}^1} \Tw(M)$ lying above $I_0$ via the projection $\Pic_{\mathbb{CP}^1} \Tw(M) \to \mathbb{CP}^1$. We have one of the following two cases. If $c_1(L)$ is $SU(2)$-invariant, then, as a consequence of Lemma \ref{thm:invariant},
\[
c_1(L) \in \bigcap_{I \in \mathbb{CP}^1} H^{1,1}(M_I) \cap H^2(M, \mathbb{Z}).
\]
Such $L$ are called \emph{hyperholomorphic}. In view of the discussion above, it's clear that the underlying topological line bundle of such $L$ admits a (unique) holomorphic structure in each induced complex structure $I$. Identifying $\Pic_{\mathbb{CP}^1} \Tw(M)$ with its image under the embedding \eqref{vlozhenie}, the connected component of $\Pic_{\mathbb{CP}^1} \Tw(M)$ containing $L$ is the line $\mathbb{CP}^1 \times \left\{c_1(L)\right\}$. On the other hand, suppose $c_1(L)$ is not $SU(2)$-invariant. In this case, $I_0$ cannot be generic. Moreover, by the proof of Proposition 2.2 in \cite{verbit1}, the intersection
\[
\Pic_{\mathbb{CP}^1} \Tw(M) \cap \left( \mathbb{CP}^1 \times \left\{ c_1(L) \right\} \right),
\]
where we again identify $\Pic_{\mathbb{CP}^1} \Tw(M)$ with its image under \eqref{vlozhenie}, is finite, hence $L$ is an isolated point of $\Pic_{\mathbb{CP}^1} \Tw(M)$.

In other words, the connected components of $\Pic_{\mathbb{CP}^1} \Tw(M)$ are either copies of $\mathbb{CP}^1$ or singletons. With this description, the analytic structure of $\Pic_{\mathbb{CP}^1} \Tw(M)$ is apparent, as well as the fact that the natural projection $\Pic_{\mathbb{CP}^1} \Tw(M) \to \mathbb{CP}^1$ is holomorphic.

Recall that stability of a vector bundle $E$ is defined as a condition on its subsheaves $F \hookrightarrow E$. Equivalently, it can be defined as a condition on its quotient sheaves $E \twoheadrightarrow Q$ (see Theorem 1.2.2 of Chapter 2 in \cite{okonek}). Recall that given a vector bundle $E$ on the twistor space $\Tw(M)$, we think of it as a family $\left\{E_I\right\}$ of vector bundles over the manifolds $\left\{M_I\right\}$ parametrized by $\mathbb{CP}^1$. In order to study stability of the bundles $E_I$, we would like to assemble all of their possible quotient sheaves into one geometric object. This is accomplished with the relative Douady Quot space construction \cite{pourcin}.

For a proper morphism of complex manifolds $X \to S$ and a vector bundle $E$ on $X$, the \emph{relative Quot space} $\Quot_S(E)$ is a complex analytic space parametrizing quotient sheaves $E_s \twoheadrightarrow Q_s$ for $s \in S$, where $E_s$ denotes the restriction of $E$ to the fibre $X_s$ of $X \to S$ over $s$.  We denote by $\Quot^1_{\mathrm{lf},S}(E)$ the open subspace of $\Quot_S(E)$ consisting of quotient sheaves $E_s \twoheadrightarrow Q_s$ with invertible kernel. In the particular case $E = \mathcal{O}_X$, we denote $\Quot^1_{\mathrm{lf},S}(E)$ by $\Dou_S(X)$ and call it the \emph{relative Douady space} of $X$ with respect to $S$. Set-theoretically, it is just the collection of effective divisors $D$ of the spaces $X_s$. The following properness result mentioned in \cite{teleman} is a consequence of Bishop's compactness theorem \cite{bishop}.

\theoremstyle{plain}
\newtheorem{douady-proper}[stable-teleman]{Theorem}
\begin{douady-proper} \label{thm:douady-proper}
Let $h$ be a Hermitian metric on a complex manifold $X$, and let $X \to S$ be a proper map onto a complex manifold $S$. Then $\forall \varepsilon > 0$ the topological subspaces
\[
\Dou_S(X)_{\le \varepsilon} := \left\{ D \in \Dou_S(X) : \Vol_h(D) \le \varepsilon \right\} \subseteq \Dou_S(X)
\]
are proper over $S$. Here, for an element $D \subseteq X_s$, $s \in S$, $\Vol_h(D)$ is the volume of $D$ with respect to the restriction of the metric $h$.
\end{douady-proper}

Let now $X \to S$ be the twistor projection $\pi : \Tw(M) \to \mathbb{CP}^1$ for a compact simple hyperk\"ahler manifold $M$, and let $E$ be a holomorphic vector bundle on $E$ of rank $r$. The relative Quot space $\Quot^1_{\mathrm{lf},\mathbb{CP}^1}(E)$ has a natural analytic map to the relative Picard group $\Pic_{\mathbb{CP}^1} \Tw(M)$, which maps every quotient $E_I \twoheadrightarrow Q_I$ to its kernel:
\[
\begin{array}{ccccc}
p & : & \Quot^1_{\mathrm{lf},\mathbb{CP}^1}(E) & \longrightarrow & \Pic_{\mathbb{CP}^1} \Tw(M) \\
&& \psi_I : E_I \twoheadrightarrow Q_I & \longmapsto & \left(I, \Ker \psi_I \right)
\end{array}
\]
It's not hard to see that, given an element $L_I \in \Pic M_I \subset \Pic_{\mathbb{CP}^1} \Tw(M)$, the set-theoretic fibre of $p$ over $L_I$ is simply
\begin{equation} \label{fibre-1}
p^{-1}(L_I) = \mathbb{P}\left(\Hom_{M_I}(L_I, E_I)\right) \cong \mathbb{P}\left(H^0(M_I, L_I^* \otimes E_I)\right).
\end{equation}

Now let $Z = \mathbb{P}(E^*)$ be the projectivization of the dual bundle of $E$ over $\Tw(M)$, thought of as a family of projectivizations $Z_I = \mathbb{P}(E_I^*)$ parametrized by $I \in \mathbb{CP}^1$, and let $\Dou_{\mathbb{CP}^1}(Z)$ be the relative Douady space of $Z$. As a first step in the proof of Theorem \ref{thm:stable-twistor}, we will identify $\Quot^1_{\mathrm{lf}, \mathbb{CP}^1}(E)$ with a certain subspace of $\Dou_{\mathbb{CP}^1}(Z)$. Just like for $\Tw(M)$, we can define the relative Picard group of $Z$ with the natural projection $\Pic_{\mathbb{CP}^1} Z \to \mathbb{CP}^1$, with the fibre over $I \in \mathbb{CP}^1$ being $\Pic Z_I = \Pic \mathbb{P}(E_I^*)$. Since $\Pic \mathbb{P}(E_I^*)$ is canonically isomorphic to $\Pic M_I \times \mathbb{Z}$, with the $\mathbb{Z}$ summand generated by the line bundle $\mathcal{O}_{Z_I}(1)$, we conclude that
\[
\Pic_{\mathbb{CP}^1} Z \cong \Pic_{\mathbb{CP}^1} \Tw(M) \times \mathbb{Z}.
\]
There is a natural analytic map
\[
\begin{array}{ccccc}
n_Z & : & \Dou_{\mathbb{CP}^1}(Z) & \longrightarrow & \Pic_{\mathbb{CP}^1} Z \\
&& D_I \subseteq M_I & \longmapsto & \left(I, \mathcal{O}_{D_I}\right),
\end{array}
\]
whose set-theoretic fibre over an element $N_I \in \Pic Z_I \subset \Pic_{\mathbb{CP}^1} Z$ is just
\begin{equation} \label{fibre-2}
n_Z^{-1}(N_I) = \mathbb{P}\left(H^0(Z_I, N_I)\right).
\end{equation}

Let $q : Z = \mathbb{P}(E^*) \to \Tw(M)$ denote the natural projection, and let $q_I : Z_I = \Pic \mathbb{P}(E_I^*)\to M_I$ be the obvious restrictions. Given a line bundle $L_I$ on $M_I$, we use the projection formula and the fact that $q_*(\mathcal{O}_Z(1)) = E$, to obtain:
\[
H^0(M_I, L_I^* \otimes E_I) \cong H^0(M_I, q_{I*}\!\left(q_I^*(L_I^*) \otimes \mathcal{O}_{Z_I}(1)\right)) \cong H^0(Z_I, q_I^*(L_I^*) \otimes \mathcal{O}_{Z_I}(1)).
\]
With the set-theoretical identifications \eqref{fibre-1}, \eqref{fibre-2} this gives us bijections of fibres $p^{-1}(L_I) \to n_Z^{-1}(q_I^*(L_I^*) \otimes \mathcal{O}_{Z_I}(1))$ for every choice of line bundle $L_I \in \Pic_{\mathbb{CP}^1}\Tw(M)$, which we can assemble into a set-theoretic embedding $\Phi : \Quot^1_{\mathrm{lf}, \mathbb{CP}^1}(E) \to \Dou_{\mathbb{CP}^1}(Z)$ that fits into the diagram
\begin{equation} \label{diagramma-relative}
\xymatrix{\Quot^1_{\mathrm{lf}, \mathbb{CP}^1}(E) \ar[d]_p \ar[r]^(0.55)\Phi & \Dou_{\mathbb{CP}^1}(Z) \ar[d]^{n_Z} \\ \Pic_{\mathbb{CP}^1}\Tw(M) \ar[r]_(0.55)a & \Pic_{\mathbb{CP}^1} Z}
\end{equation}
Here, $a : \Pic_{\mathbb{CP}^1} \Tw(M) \to \Pic_{\mathbb{CP}^1} Z$ is just the map $(I, L_I) \mapsto (I, q_I^*(L_I^*) \otimes \mathcal{O}_{Z_I}(1))$. We would like to verify that $\Phi$ is actually analytic.

\theoremstyle{plain}
\newtheorem{douady-embedding}[stable-teleman]{Proposition}
\begin{douady-embedding} \label{thm:douady-embedding}
The map
\[
\Phi : \Quot^1_{\mathrm{lf}, \mathbb{CP}^1}(E) \stackrel{\sim}{\to} n_Z^{-1}(a(\Pic_{\mathbb{CP}^1}\Tw(M))) \subseteq \Dou_{\mathbb{CP}^1}(Z)
\]
is a complex analytic isomorphism.
\end{douady-embedding}

\begin{proof}
The proof closely follows the argument of Teleman in Proposition 2.3 of \cite{teleman}. The spaces $\Quot^1_{\mathrm{lf}, \mathbb{CP}^1}(E)$, $n_Z^{-1}(a(\Pic_{\mathbb{CP}^1}\Tw(M)))$ represent contravariant functors on the category of complex analytic spaces over $\mathbb{CP}^1$, which we denote by $\mathcal{Q}uot^1_{\mathrm{lf}, \mathbb{CP}^1}(E)$, $\mathcal{D}$, respectively, and the argument consists of exhibiting $\Phi$ as an isomorphism between these two functors. They are defined as follows. For a complex analytic space $T$ and an analytic map $g : T \to \mathbb{CP}^1$,
\[
\mathcal{Q}uot^1_{\mathrm{lf}, \mathbb{CP}^1}(E)(T) := \left\{\textrm{quotients } E_T \to Q \to 0 \textrm { on } \Tw(M)_T = \Tw(M) \times_{\mathbb{CP}^1} T : \right.
\]
\[
\left. Q \textrm{ is flat over } T \textrm{ and } E_T \to Q \textrm{ has invertible kernel} \right\},
\]
where $E_T$ is the pullback of $E$ via the projection $\Tw(M)_T = \Tw(M) \times_{\mathbb{CP}^1} T \to \Tw(M)$, and
\[
\mathcal{D}(T) := \left\{\textrm{divisors } D \subseteq Z_T = Z \times_{\mathbb{CP}^1} T : \right.
\]
\[
\left. \mathcal{O}_D \textrm{ is flat over } T \textrm{ and } \forall t \in T, \ \mathcal{O}(D_t) \in a(\Pic M_{g(t)})\right\},
\]
where $D_t$ is the restriction of $D$ to the fibre $Z_{g(t)}$ of $Z_T = Z \times_{\mathbb{CP}^1} T \to T$ over $t \in T$, and the morphism $a : \Pic M_{g(t)} \to \Pic Z_{g(t)}$ comes from the diagram \eqref{diagramma-relative}.

Now fix $g : T \to \mathbb{CP}^1$. We will construct a bijection between $\mathcal{Q}uot^1_{\mathrm{lf}, \mathbb{CP}^1}(E)(T)$ and $\mathcal{D}(T)$ in a very similar way to $\Phi$. Let $q_T : Z_T \cong \mathbb{P}(E_T^*) \to \Tw(M)_T$ be the natural projection, and let $a_T : \Pic \Tw(M)_T \to \Pic Z_T$ be the map $L \mapsto q_T^*(L^*) \otimes \mathcal{O}_{Z_T}(1)$. Similarly to the morphisms $p$ and $n_Z$ in the diagram \eqref{diagramma-relative}, there are canonical maps
\[
\begin{array}{ccccc}
p_T & : & \mathcal{Q}uot^1_{\mathrm{lf}, \mathbb{CP}^1}(E)(T) & \longrightarrow & \Pic \Tw(M)_T \\
&& \psi : E_T \twoheadrightarrow Q & \longmapsto & \Ker \psi
\end{array}
\]
with fibre over  $L \in \Pic \Tw(M)_T$ equal to
\[
p_T^{-1}(L) = \left\{ \left[\phi\right] \in \mathbb{P}\left(\Hom(L, E_T)\right) : \phi \textrm{ sheaf monomorphism, quotient flat over } T \right\},
\]
and
\[
\begin{array}{ccccc}
n_{Z_T} & : & \mathcal{D}(T) & \longrightarrow & \Pic Z_T \\
&& D \subseteq Z_T & \longmapsto & \mathcal{O}(D)
\end{array}
\]
with fibre over $N \in a_T(\Pic \Tw(M)_T) \subset \Pic Z_T$ equal to
\[
n_{Z_T}^{-1}(N) = \left\{ \left[\psi\right] \in \mathbb{P}\left(\Hom(\mathcal{O}_{Z_T}, N)\right) : \psi \textrm{ sheaf monomorphism, quotient flat over } T \right\}.
\]
Now given $L \in \Pic \Tw(M)_T$, using the projection formula and $(q_T)_*(\mathcal{O}_{Z_T}(1)) = E_T$,
\[
\Hom(L, E_T) = H^0(\Tw(M)_T, L^* \otimes E_T) \cong H^0(Z_T, q_T^*(L^*) \otimes \mathcal{O}_{Z_T}(1)) = \Hom(\mathcal{O}_{Z_T}, a_T(L)).
\]
Just as above, we would like to conclude that this gives us bijections of fibres $p_T^{-1}(L) \longleftrightarrow n_{Z_T}^{-1}(a_T(L))$ for every choice of line bundle $L \in \Pic \Tw(M)_T$, which assemble into a bijection $\Phi_T$ that fits into the diagram
\[
\xymatrix{\mathcal{Q}uot^1_{\mathrm{lf}, \mathbb{CP}^1}(E)(T) \ar[d]_{p_T} \ar[r]^(0.7){\Phi_T} & \mathcal{D}(T) \ar[d]^{n_{Z_T}} \\ \Pic \Tw(M)_T \ar[r]_(0.55){a_T} & \Pic Z_T}
\]

However, two things need to be verified. First, as non-reduced and reducible spaces have nonzero morphisms $L \to E_T$ that are not sheaf monomorphisms, one has to check that in our correspondence, sheaf monomorphisms in $\Hom(L, E_T)$ get mapped to sheaf monomorphisms in $\Hom(\mathcal{O}_{Z_T}, q_T^*(L^*) \otimes \mathcal{O}_{Z_T}(1))$. This is a local statement on $\Tw(M)_T$, which is easily verified using trivializations of $L$ and $E_T$. The second verification one has to make is that the flatness conditions for $p_T^{-1}(L)$ and $n_{Z_T}^{-1}(q_T^*(L^*) \otimes \mathcal{O}_{Z_T}(1))$ are equivalent. So let $Q$ be the quotient of a monomorphism $L \to E_T$, and let $Q'$ be the quotient of the corresponding monomorphism $\mathcal{O}_{Z_T} \to q_T^*(L^*) \otimes \mathcal{O}_{Z_T}(1)$. Choosing a point $x = (m, t) \in \Tw(M)_T = \Tw(M) \times_{\mathbb{CP}^1} T$, by the local flatness criterion (see \cite{eisenbud}, Theorem 6.8), $Q$ is $T$-flat at $(m, t)$ if and only if $\Tor_1^{\mathcal{O}_t} (\mathbb{C}_t, Q_{(m,t)}) = 0$. Since $\pi_T : \Tw(M)_T \to T$ is a flat morphism (a consequence of the fact that $\pi : \Tw(M) \to \mathbb{CP}^1$ is flat), we have
\[
\Tor_1^{\mathcal{O}_t} (\mathbb{C}_t, Q_{(m,t)}) = \Tor_1^{\mathcal{O}_{(m,t)}}(\mathbb{C}_t \otimes_{\mathcal{O}_t} \mathcal{O}_{(m,t)}, Q_{(m,t)}),
\]
and the latter is just the stalk at $(m,t)$ of the sheaf $\mathcal{T}or_1(\mathcal{O}_{\pi_T^{-1}(t)}, Q) = \mathcal{T}or_1(\mathcal{O}_{M_{g(t)}}, Q)$. So the flatness of $Q$ is equivalent to the vanishing of the sheaves $\mathcal{T}or_1(\mathcal{O}_{M_{g(t)}}, Q)$ for every $t \in T$, which in turn is equivalent to the injectivity of the sheaf morphism
\[
L_t := \left.L\right|_{\pi_T^{-1}(t)} \longrightarrow \left.(E_T)\right|_{\pi_T^{-1}(t)} \cong E_{g(t)}
\]
for every $t$. By an entirely analogous argument, $Q'$ is flat if and only if the induced sheaf morphism
\[
\left.\mathcal{O}_{Z_T}\right|_{\pi_T^{-1}(t)} \cong \mathcal{O}_{Z_{g(t)}} \longrightarrow q_{g(t)}^*(L_t^*) \otimes \mathcal{O}_{Z_{g(t)}}(1)
\]
is injective for all $t$. The equivalence of the two conditions is shown exactly as the corresponding statement for $\Hom(L, E_T)$ and $\Hom(\mathcal{O}_{Z_T}, q_T^*(L^*) \otimes \mathcal{O}_{Z_T}(1))$.
\end{proof}

We would now like to apply Proposition \ref{thm:douady-embedding} to translate Theorem \ref{thm:douady-proper} from a statement about properness of subsets of $\Dou_{\mathbb{CP}^1}(Z)$ into a statement about properness of subsets of $\Quot^1_{\mathrm{lf}, \mathbb{CP}^1}(E)$.

\theoremstyle{plain}
\newtheorem{degree-proper}[stable-teleman]{Proposition}
\begin{degree-proper} \label{thm:degree-proper}
Let $g$ denote the hyperk\"ahler metric on $M$. For any $d \in \mathbb{R}$, the subspaces
\[
\Quot^1_{\mathrm{lf}, \mathbb{CP}^1}(E)_{\ge d}, \, \Quot^1_{\mathrm{lf}, \mathbb{CP}^1}(E)_{> d} \subseteq \Quot^1_{\mathrm{lf}, \mathbb{CP}^1}(E),
\]
defined by the inequalities $\deg_g(L_I) \ge d$, resp. $\deg_g(L_I) > d$, are complex analytic and proper over $\mathbb{CP}^1$.
\end{degree-proper}

\begin{proof}
Recall that we have maps
\[
\begin{array}{ccccc}
\Quot^1_{\mathrm{lf}, \mathbb{CP}^1}(E) & \stackrel{p}{\longrightarrow} & \Pic_{\mathbb{CP}^1} \Tw(M) & \stackrel{\deg_g}{\longrightarrow} & \mathbb{R} \\ \psi :E_I \twoheadrightarrow Q_I & \longmapsto & (I, L_I := \Ker \psi) & \longmapsto & \deg_g(L_I)
\end{array}
\]
As we saw, the connected components of $\Pic_{\mathbb{CP}^1} \Tw(M)$ are either isolated points or copies of $\mathbb{CP}^1$, and on the latter the degree map is constantly zero, by virtue of Lemma \ref{thm:degr}. In particular, the degree map is locally constant on $\Pic_{\mathbb{CP}^1} \Tw(M)$, hence it is also locally constant on $\Quot^1_{\mathrm{lf}, \mathbb{CP}^1}(E)$. It follows at once that both $\Quot^1_{\mathrm{lf}, \mathbb{CP}^1}(E)_{\ge d}$ and $\Quot^1_{\mathrm{lf}, \mathbb{CP}^1}(E)_{> d}$ are unions of connected components of $\Quot^1_{\mathrm{lf}, \mathbb{CP}^1}(E)$, hence they are analytic. It remains to show that they are compact.

The rest of the proof closely follows the argument of Teleman on page 9 of \cite{teleman}. Let $r$ denote the rank of $E$ and $n$ the complex dimension of $M$. Recall that the hyperk\"ahler metric $g$ induces a natural metric on the twistor space $\Tw(M)$, which we will also denote by $g$, abusing the notation slightly. Choose an arbitrary Hermitian metric $h$ on $E$. The Chern connection of $(E, h)$ induces an Ehresmann connection on the projective bundle $q : Z = \mathbb{P}(E^*) \to \Tw(M)$, that is, a horizontal subbundle $HZ \subseteq TZ$ such that there is a direct sum decomposition
\[
TZ = HZ \oplus VZ,
\]
where $VZ$ is the vertical tangent bundle of $q : Z \to \Tw(M)$. Because the Chern connection of $(E, h)$ is compatible with the holomorphic structure of $E$, the distribution $HZ \subseteq TZ$ is preserved by the almost-complex structure of $Z$. On the other hand, the metric $h$ on $E$ induces a natural Hermitian metric on the vertical tangent bundle $VZ$ on $Z$, which is just the Fubini-Study metric on the fibres $q^{-1}(x) = \mathbb{P}(E_x^*) \cong \mathbb{P}^{r-1}$. We will denote by $\omega_{FS}$ the corresponding Hermitian form, thought of as a real vertical $(1,1)$-form on $Z$. It is now easy to see that if $\omega$ denotes the Hermitian form of $g$ on $\Tw(M)$, 
\[
\Omega := q^*(\omega) + \omega_{FS}
\]
is a real positive $(1,1)$-form on $Z$ such that $TZ = HZ \oplus VZ$ becomes an orthogonal direct sum in the corresponding metric $G$ on $Z$. Letting $I \in \mathbb{CP}^1$, the restriction of $G$ to the submanifold $Z_I \subseteq Z$, as in the diagram
\[
\xymatrix{ Z_I = \mathbb{P}(E_I^*) \ar@{^{(}->}[r] \ar[d]_{q_I} & Z = \mathbb{P}(E^*) \ar[d]^q \\ M_I \ar@{^{(}->}[r] & \Tw(M), }
\]
will be denoted by $G_I$. Its Hermitian form is
\[
\Omega_I := q_I^*(\omega_I) + \omega_{FS},
\]
where $\omega_I$ is the K\"ahler form on $M_I$.

Now fix $I \in \mathbb{CP}^1$ and let $L_I$ be a holomoprhic line bundle on $M_I$. We want to relate the degree of $L_I$ with respect to $g$ to the degree of $q_I^*(L_I)$ with respect to a Gauduchon metric in the conformal class of $G_I$. By Theorem \ref{thm:kobayashi-hitchin}, there exists a $g$-Hermitian-Einstein metric $\gamma$ on $L_I$, and by Proposition \ref{thm:einst-slope}, the curvature $R^\gamma$ of the Chern connection on $(L_I, \gamma)$ satisfies the equation
\[
\frac{\sqrt{-1}}{2\pi}R^\gamma \wedge \omega_I^{n-1} = \frac{\deg_g(L_I)}{n! \Vol_g(M)}\omega_I^n.
\]
We now verify that the metric $q_I^*(\gamma)$ on $q_I^*(L_I)$ is $G_I$-Hermitian-Einstein, and its Einstein constant is proportional to $\deg_g(L_I)$. We will use the fact that $\omega_{FS}^k = 0$ for $k \ge r$ on $Z_I$.
\[
\left(\frac{\sqrt{-1}}{2\pi}q_I^*(R^\gamma)\right) \wedge \Omega_I^{n+r-2} = \binom{n+r-2}{r-1} \, q_I^*\left(\frac{\sqrt{-1}}{2\pi}R^\gamma \wedge \omega_I^{n-1}\right) \wedge \omega_{FS}^{r-1} =
\]
\[
= \frac{(n+r-2)!}{(r-1)! (n-1)!} \frac{\deg_g(L_I)}{n! \Vol_g(M)} q_I^*(\omega_I^n) \wedge \omega_{FS}^{r-1} =
\]
\[
= \frac{1}{n+r-1}\left(\frac{\deg_g(L_I)}{(n-1)! \Vol_g(M)}\right) \Omega_I^{n+r-1},
\]
where we have used the fact that
\[
\Omega_I^{n+r-1} = \binom{n+r-1}{r-1} \, q_I^*(\omega_I^n) \wedge \omega_{FS}^{r-1}.
\]
We have thus shown that the $G_I$-Hermitian-Einstein constant of the line bundle $q_I^*(L_I)$ on $Z_I$ is proportional to $\deg_g(L_I)$ (and in fact the constant of proportionality does not depend on the complex structure $I$).

If we now let $I \in \mathbb{CP}^1$ vary, then since $G_I$ depends smoothly on $I$, we can choose a family of Gauduchon metrics $\left\{G_I'\right\}$ on $\left\{Z_I\right\}$ such that $G_I'$ is in the conformal class of $G_I$ and $G_I'$ depends smoothly on $I$. Let the function $C_1 : \mathbb{CP}^1 \to \mathbb{R}$ be defined by
\[
C_1(I) := \deg_{G_I'}(\mathcal{O}_{Z_I}(1)).
\]
By Proposition 1.3.5 of \cite{lubke-teleman}, there is a function $C_2 : \mathbb{CP}^1 \to \mathbb{R}^{>0}$ such that for any $I \in \mathbb{CP}^1$ and any nontrivial line bundle $N_I \in \Pic Z_I$ with a nonzero section $s \in H^0(Z_I, N_I)$,
\[
\Vol_{G_I'} \left\{s = 0\right\} = C_2(I) \cdot \deg_{G_I'} N_I.
\]
Finally, by the result in the previous paragraph and Proposition \ref{thm:degree-proportional}, there is also a function $C_3 : \mathbb{CP}^1 \to \mathbb{R}^{>0}$ such that for any $I \in \mathbb{CP}^1$ and $L_I \in \Pic M_I$,
\[
\deg_{G_I'} (q_I^*(L_I)) = C_3(I) \cdot \deg_g(L_I).
\]
All three functions $C_1, C_2, C_3$ are continuous on $\mathbb{CP}^1$ by continuity of the family $\left\{G_I'\right\}$.

Recall that we have a map
\[
\Phi : \Quot^1_{\mathrm{lf},\mathbb{CP}^1}(E) \longrightarrow \Dou_{\mathbb{CP}^1}(Z),
\]
which is a complex analytic isomorphism of $\Quot^1_{\mathrm{lf},\mathbb{CP}^1}(E)$ with a union of connected components in $\Dou_{\mathbb{CP}^1}(Z)$. Fixing $I \in \mathbb{CP}^1$, a line bundle $L_I \in \Pic M_I$ and a nonzero sheaf monomorphism $\varphi_I : L_I \to E_I$ thought of as an element of $\Quot^1_{\mathrm{lf}, \mathbb{CP}^1}(E)$, we have
\[
\Vol_{G_I'} \Phi(\left[\varphi_I\right]) = C_2(I) \cdot \deg_{G_I'} (\mathcal{O}_{Z_I}(1) \otimes q_I^*(L_I^*)) =
\]
\[
= C_2(I) \cdot \deg_{G_I'}(\mathcal{O}_{Z_I}(1)) - C_2(I) \cdot C_3(I) \cdot \deg_g(L_I).
\]
\[
= C_2(I) \cdot C_1(I) - C_2(I) \cdot C_3(I) \cdot \deg_g(L_I).
\]
Letting $\varepsilon > 0$, we know by Theorem \ref{thm:douady-proper} that the subset
\[
\Phi^{-1}\left(\Dou_{\mathbb{CP}^1}(Z)_{\le \varepsilon}\right) = \left\{ \varphi_I : L_I \to E_I \left| \ \deg_g(L_I) \ge \frac{C_2(I)C_1(I) - \varepsilon}{C_2(I)C_3(I)} \right.\right\} \subseteq \Quot^1_{\mathrm{lf},\mathbb{CP}^1}(E)
\]
is proper over $\mathbb{CP}^1$, hence in particular compact. Choosing $\varepsilon \gg 0$ large enough so that both $\Quot^1_{\mathrm{lf}, \mathbb{CP}^1}(E)_{\ge d}$ and $\Quot^1_{\mathrm{lf}, \mathbb{CP}^1}(E)_{> d}$ are subsets of $\Phi^{-1}\left(\Dou_{\mathbb{CP}^1}(Z)_{\le \varepsilon}\right)$, we conclude that both are compact. We are done.
\end{proof}

To go ahead with the proof of Theorem \ref{thm:stable-twistor}, we need the following version of the Pl\"ucker embedding for vector bundles. Let $Y$ be a complex manifold and $E$ a holomorphic vector bundle on $Y$ of rank $r$. For $1 \le s \le r-1$, the \emph{cone of exterior monomials} $C_s(E) \subseteq \Lambda^s E$ is a cone subbundle of $\Lambda^s E$, which over a point $y \in Y$ consists of elements that can be written in the form $v_1 \wedge \ldots \wedge v_s$, for $v_i \in E_y$. We have the following correspondence.
\begin{equation} \label{correspondence}
\left\{\begin{array}{c} \textrm{Subsheaves } \mathcal{F} \hookrightarrow E \\ \textrm{of rank } s \end{array} \right\} \substack{\longrightarrow \\ \longleftarrow} \left\{\begin{array}{c} \textrm{Line subsheaves } L \hookrightarrow \Lambda^s E \\ \textrm{with image in } C_s(E) \subseteq \Lambda^s E \end{array} \right\}
\end{equation}
Given a rank $s$ subsheaf $\mathcal{F} \subset E$, the corresponding line subsheaf is just $\det \mathcal{F} \subset \Lambda^s E$. On the other hand, given a line subsheaf $L \subset \Lambda^s E$ with image in $C_s(E)$, we can tensor the sheaf monomorphism $L \hookrightarrow \Lambda^s E$ with $\Lambda^{s-1} E^*$ and take the composition
\[
L \otimes \Lambda^{s-1} E^* \longrightarrow \Lambda^s E \otimes \Lambda^{s-1} E^* \longrightarrow E,
\]
where the second arrow is tensor contraction. Taking the maximal normal extension (see \cite{okonek}, page 80) of the image of this morphism in $E$, we get a subsheaf $\mathcal{F} \subset E$ of rank $s$. If one requires the subsheaves on both sides of the correspondence \eqref{correspondence} to have torsion-free quotients, one can check that the arrows become set-theoretical inverses. With this construction, the following result becomes apparent.

\theoremstyle{plain}
\newtheorem{stability-cone}[stable-teleman]{Proposition}
\begin{stability-cone} \label{thm:stability-cone}
Let $Y$ be a complex manifold with a Gauduchon metric $g$, and let $E$ be a holomorphic rank $r$ vector bundle over $Y$. The following conditions are equivalent:
\begin{enumerate}
\item[(i)] $E$ is $g$-stable ($g$-semi-stable).
\item[(ii)] For every $1 \le s \le r-1$, and any non-trivial morphism $\varphi : L \to \Lambda^s E$, where $L$ is a line bundle and $\Imag(\varphi) \subseteq C_s(E)$, one has
\[
\deg_g L < s \cdot \mu_g(E) \ (\textrm{resp. } \deg_g L \le s \cdot \mu_g(E)).
\]
\end{enumerate}
\end{stability-cone}

\begin{proof}
Proposition 2.15 in \cite{teleman}.
\end{proof}

\begin{proof}[Proof of Theorem \ref{thm:stable-twistor}]
Fix $1 \le s \le r -1$. Recall that the set $S_0 \subseteq {{}S}^2 \cong \mathbb{CP}^1$ of generic complex structures on $M$ is dense in $\mathbb{CP}^1$ by Proposition \ref{thm:dense}. For $I \in S_0$, $\mu_g(E_I) = 0$ by Lemma \ref{thm:degr}, so by continuity, $\mu_g(E_I) = 0$ for all $I \in \mathbb{CP}^1$, and similarly, $\mu_g(\Lambda^s E_I) = 0$. By Proposition \ref{thm:degree-proper}, the subpaces
\[
\Quot^1_{\mathrm{lf}, \mathbb{CP}^1}(\Lambda^s E)_{\ge d}, \ \Quot^1_{\mathrm{lf}, \mathbb{CP}^1}(\Lambda^s E)_{> d} \subseteq \Quot^1_{\mathrm{lf}, \mathbb{CP}^1}(\Lambda^s E)
\]
are analytic and proper over $\mathbb{CP}^1$, hence their projections in $\mathbb{CP}^1$
\[
\mathbb{CP}^1_{\ge d} = \left\{I : \exists L_I \in \Pic M_I \textrm{ and } \varphi_I : L_I \hookrightarrow \Lambda^s E_I \textrm{ such that } \deg_g L_I \ge d\right\},
\]
\[
\mathbb{CP}^1_{> d} = \left\{I : \exists L_I \in \Pic M_I \textrm{ and } \varphi_I : L_I \hookrightarrow \Lambda^s E_I \textrm{ such that } \deg_g L_I > d\right\}
\]
are Zariski closed. Let $\Quot^1_{\textrm{lf},\mathbb{CP}^1}(E)^s$ be the closed analytic subspace of $\Quot^1_{\textrm{lf},\mathbb{CP}^1}(\Lambda^s E)$ consisting of equivalence classes of sheaf monomorphisms $\varphi_I : L_I \to \Lambda^s E_I$ with $\Imag(\varphi_I) \subseteq C_s(E_I)$. Then the intersections
\[
\Quot^1_{\mathrm{lf}, \mathbb{CP}^1}(\Lambda^s E)_{\ge d} \cap \Quot^1_{\mathrm{lf}, \mathbb{CP}^1}(E)^s, \ \Quot^1_{\mathrm{lf}, \mathbb{CP}^1}(\Lambda^s E)_{> d} \cap \Quot^1_{\mathrm{lf}, \mathbb{CP}^1}(E)^s
\]
are again analytic and proper over $\mathbb{CP}^1$, hence their projections in $\mathbb{CP}^1$
\[
\left(\mathbb{CP}^1_{\ge d}\right)^s = \left\{I : \exists \varphi_I : L_I \hookrightarrow \Lambda^s E_I \textrm{ with } \Imag(\varphi_I) \subseteq C_s(E_I) \textrm{ and such that } \deg_g L_I \ge d\right\},
\]
\[
\left(\mathbb{CP}^1_{> d}\right)^s = \left\{I : \exists \varphi_I : L_I \hookrightarrow \Lambda^s E_I \textrm{ with } \Imag(\varphi_I) \subseteq C_s(E_I) \textrm{ and such that }  \deg_g L_I > d\right\}
\]
are again Zariski closed. It only remains to observe that, by Proposition \ref{thm:stability-cone},
\[
\left(\mathbb{CP}^1\right)^{\textrm{st}} = \mathbb{CP}^1 \setminus \bigcup_{1 \le s \le r-1} \left(\mathbb{CP}^1_{\ge 0}\right)^s, \, \left(\mathbb{CP}^1\right)^{\textrm{sst}} = \mathbb{CP}^1 \setminus \bigcup_{1 \le s \le r-1} \left(\mathbb{CP}^1_{> 0}\right)^s.
\]

\end{proof}

\section{Irreducible bundles on $\mathrm{Tw}(M)$ and fibrewise stability} \label{sect:main}

Recall that an irreducible vector bundle is one that does not have proper subsheaves of lower rank, while a generically fibrewise stable bundle $E$ on the twistor space $\Tw(M)$ of a hyperk\"ahler manifold $M$ is one whose restriction $E_I$ to the fibre $M_I$ of the twistor projection $\pi : \Tw(M) \to \mathbb{CP}^1$ is stable for all $I \in \mathbb{CP}^1$, except perhaps finitely many. The main result of the article follows.

\theoremstyle{plain}
\newtheorem{result}{Theorem}[section]
\begin{result} \label{thm:result}
Let $M$ be a compact simple hyperk\"ahler manifold and let $E$ be a holomorphic vector bundle on the twistor space $\Tw(M)$. If $E$ is generically fibrewise stable, then it is irreducible. The converse is true for vector bundles of rank 2 and 3, as well as for bundles $E$ of general rank which are simple over some fibre of the projection $\pi : \Tw(M) \to \mathbb{CP}^1$, i.e. $\Hom_{M_I}(E_I, E_I) = \mathbb{C}$ for some $I \in \mathbb{CP}^1$.
\end{result}

\begin{proof}[Proof of forward implication and converse for the cases $\rk E = 2, 3$]
The forward implication is due to Kaledin and Verbitsky (\cite{kaled-verbit}, Lemma 7.3).  Suppose $E$ is generically fibrewise stable. Since the set of generic induced complex structures $S_0$ is dense in $\mathbb{CP}^1$ by Proposition \ref{thm:dense}, we can always choose $I \in S_0$ such that $E_I$ is stable. In fact, by Lemma \ref{thm:degr}, it is irreducible, since any proper subsheaf of $E_I$ of lower rank would destabilize $E_I$, both having slope 0. Any subsheaf $\mathcal{F} \subseteq E$ on $\Tw(M)$ is torsion-free, hence its singularity set has codimension $\ge 2$, so in particular its restriction to $M_I$ is a subsheaf $\mathcal{F}_I \subseteq E_I$ of the same rank as $\mathcal{F}$. It follows that either $\rk(\mathcal{F}) = 0$, so that $\mathcal{F} = 0$, or $\rk(\mathcal{F}) = \rk (E)$. Thus $E$ is irreducible. Observe that in our proof we have only used that $E_I$ is stable for a single generic complex structure $I \in S_0 \subseteq \mathbb{CP}^1$, which is consistent with the results of the previous section.

We now prove the converse for the cases $\rk E = 2$ and $3$. Let $E$ be a vector bundle of rank 2 on $\Tw(M)$, and suppose $E$ is not generically fibrewise stable. Then, by Theorem \ref{thm:stable-twistor}, it actually follows that $E_I$ is non-stable for all $I \in \mathbb{CP}^1$, i.e. the map
\[
\Quot^1_{\textrm{lf}, \mathbb{CP}^1}(E)_{\ge 0} \longrightarrow \mathbb{CP}^1
\]
is surjective. Since this map is analytic and proper, we conclude that there is a connected component in $\Quot^1_{\textrm{lf}, \mathbb{CP}^1}(E)_{\ge 0}$ which projects onto $\mathbb{CP}^1$. By the set-theoretic description of the relative Quot space $\Quot^1_{\textrm{lf}, \mathbb{CP}^1}(E)$ and the relative Picard group $\Pic_{\mathbb{CP}^1} \Tw(M)$ given in the previous section, this means that there exists a (topological) complex line bundle $L$ on $M$ with $SU(2)$-invariant first Chern class $c_1(L)$ such that for every induced complex structure $I \in \mathbb{CP}^1$, $L$ admits a unique holomorphic structure $L_I$ over $M_I$, and there exist nontrivial morphisms $L_I \to E_I$. Moreover, just as in the proof of Proposition \ref{thm:divisors}, one can construct a Hermitian metric on $L$ and a Hermitian connection $\nabla$ whose curvature is $SU(2)$-invariant, and so as a consequence of Lemma \ref{thm:invariant}, the (0,1)-part of $\nabla$ with respect to $I$ induces the holomorphic structure $L_I$, for all $I$. Taking the pullback bundle and connection $(\sigma^*\!L, \sigma^*\nabla)$ along the (non-holomorphic) twistor projection $\sigma : \Tw(M) \to M$, it's not hard to check that the $(0,1)$-part of $\sigma^*\nabla$ defines a holomorphic structure on $\sigma^*\!L$. We will denote this holomorphic line bundle on $\Tw(M)$ by $L$ as well, as this should not cause any confusion. Note that the restriction of this $L$ to the fibre $M_I = \pi^{-1}(I)$ is just $L_I$.

In sum, we have a holomorphic line bundle $L$ on $\Tw(M)$ such that for every $I \in \mathbb{CP}^1$, 
\[
\dim \Hom_{M_I}(L_I, E_I) = \dim H^0(M_I, L_I^* \otimes E_I) \ge 1,
\]
where $L_I$ is the restriction of $L$ to the fibre $M_I$ of the holomorphic twistor projection $\pi : \Tw(M) \to \mathbb{CP}^1$. It follows that the pushforward sheaf $\pi_*(L^* \otimes E)$ is nonzero. Moreover, it is torsion-free because $L^* \otimes E$ is, and since torsion-free sheaves on $\mathbb{CP}^1$ are locally free, we conclude that $\pi_*(L^* \otimes E)$ is a nonzero vector bundle on $\mathbb{CP}^1$. Let
\[
0 \longrightarrow K \longrightarrow \pi_*(L^* \otimes E)
\]
be any line subsheaf. Taking its pullback along $\pi$ and composing with the evaluation map $\pi^*\pi_*(L^* \otimes E) \to L^* \otimes E$, we get a nonzero morphism
\[
\pi^* K \longrightarrow \pi^*\pi_*(L^* \otimes E) \longrightarrow L^* \otimes E.
\]
Tensoring this composition with $L$, we get a line subsheaf of $E$, hence $E$ is not irreducible.

Now let $E$ be a vector bundle of rank 3 on $\Tw(M)$, and suppose $E$ is not generically fibrewise stable. Again, by Theorem \ref{thm:stable-twistor}, $E_I$ admits a destabilizing subsheaf $\forall I \in \mathbb{CP}^1$. In the notation of the proof of Theorem \ref{thm:stable-twistor}, we have
\[
\mathbb{CP}^1 = \left(\mathbb{CP}_{\ge 0}^1\right)^1 \cup \left(\mathbb{CP}_{\ge 0}^1\right)^2,
\]
and since these subsets are Zariski closed, it follows that one of them is equal to the whole $\mathbb{CP}^1$. If $\left(\mathbb{CP}_{\ge 0}^1\right)^1 = \mathbb{CP}^1$, then a repeat of the argument for the case $\rk E = 2$ gives a line subsheaf of $E$. In case $\left(\mathbb{CP}_{\ge 0}^1\right)^2 = \mathbb{CP}^1$, observing that  the cone subbundle of exterior monomials $C_2(E) \subseteq \Lambda^2 E$ is equal to the whole $\Lambda^2 E$ for a rank 3 vector bundle, we repeat the same argument again with $E$ replaced by $\Lambda^2 E$ to conclude the existence of a line subsheaf of $\Lambda^2 E$ on $\Tw(M)$, which by the correspondence \eqref{correspondence} gives a rank 2 subsheaf of $E$.
\end{proof}

Before proving the converse for the case that $E$ is simple over a fibre of the projection $\pi : \Tw(M) \to \mathbb{CP}^1$, let us see where the generalization of our argument for the case $\rk E = 3$ to bundles of general rank $r$ breaks down. Let $E$ be irreducible and assume for contradiction that $E$ is not generically fibrewise stable. Theorem \ref{thm:stable-twistor} again gives us
\[
\mathbb{CP}^1 = \left(\mathbb{CP}_{\ge 0}^1\right)^1 \cup \left(\mathbb{CP}_{\ge 0}^1\right)^2 \cup \ldots \cup\left(\mathbb{CP}^1_{\ge 0}\right)^{r-1}.
\]
Since the subsets on the right are Zariski closed, one of them must be equal to $\mathbb{CP}^1$, say $\left(\mathbb{CP}_{\ge 0}^1\right)^s$ for some $1 \le s \le r-1$. Again there is some line bundle $L$ on $\Tw(M)$ such that for all $I \in \mathbb{CP}^1$ there exist nontrivial morphisms $L_I \to \Lambda^s E_I$ over $M_I$ with image lying inside the cone of exterior monomials $C_s(E_I) \subseteq \Lambda^s E_I$. Hence $\pi_*(L^* \otimes \Lambda^s E)$ is a nonzero vector bundle on $\mathbb{CP}^1$, and by Grauert's theorem (Theorem 10.5.5 in \cite{grauert-remmert}), for all $I \in \mathbb{CP}^1$, except possibly finitely many, its fibre has the form
\[
\pi_*(L^* \otimes \Lambda^s E)_I \cong H^0(M_I, L_I^* \otimes \Lambda^s E_I) = \Hom_{M_I}(L_I, \Lambda^s E_I).
\]
However, unless $s = 1$ or $r-1$, it's no longer true that $C_s(E_I) = \Lambda^s E_I$, so taking an arbitrary line subsheaf
\[
0 \longrightarrow K \longrightarrow \pi_*(L^* \otimes \Lambda^s E)
\]
on $\mathbb{CP}^1$ no longer guarantees that the corresponding composition
\[
\pi^* K \longrightarrow \pi^*\pi_*(L^* \otimes \Lambda^s E) \longrightarrow L^* \otimes \Lambda^s E.
\]
on $\Tw(M)$ will take values in $C_s(E)$, so it will not in general give a rank $s$ subsheaf of $E$. Thus, while there are monomorphisms $L_I \hookrightarrow \Lambda^s E_I$ with values in $C_s(E_I)$ for all $I \in \mathbb{CP}^1$, it's not apparent that they can be ``glued'' into a global morphism over $\Tw(M)$.

To describe this problem slightly differently, take the projectivization of the vector bundle $\pi_*(L^* \otimes \Lambda^s E)$ on $\mathbb{CP}^1$, and note that there is a 1-to-1 correspondence between line subbundles of $\pi_*(L^* \otimes \Lambda^s E)$ and sections of the projection $v : \mathbb{P}(\pi_*(L^* \otimes \Lambda^s E)) \to \mathbb{CP}^1$. By Grauert's theorem, the generic fibre of $v$ looks like
\[
v^{-1}(I) = \mathbb{P}(\Hom_{M_I}(L_I, \Lambda^s E_I)),
\]
and since any scalar multiple of an element of $\Hom_{M_I}(L_I, \Lambda^s E_I)$ taking values in $C_s(E_I)$ clearly also takes values in $C_s(E_I)$, we get a well-defined closed analytic subset
\[
\xymatrix{Y := \left\{\textrm{Classes of maps } L_I \hookrightarrow \Lambda^s E_I \textrm{ with image in }  C_s(E_I)\right\} \ar@{^{(}->}[r] \ar[dr]_{u} & \mathbb{P}(\pi_*(L^* \otimes \Lambda^s E)) \ar[d]^v \\ & \mathbb{CP}^1,}
\]
where the map $u$ is surjective. The problem then reduces to finding a section of the map $u : Y \to \mathbb{CP}^1$, which would give a line subbundle $K \hookrightarrow \pi_*(L^* \otimes \Lambda^s E)$, from which one can construct a rank $s$ subsheaf of $E$ on $\Tw(M)$, as described in the previous paragraph. Note that at this point it becomes a purely algebraic problem, since $\mathbb{P}(\pi_*(L^* \otimes \Lambda^s E))$ is a projective algebraic variety, and so is $Y$, by Chow's theorem. Unfortunately, we don't have any information about the structure of $Y$, so we cannot assume the existence of a section of $u$. However, we have the following algebraic result.

\theoremstyle{plain}
\newtheorem{multisection}[result]{Lemma}
\begin{multisection} \label{thm:multisect}
Let $u : Y \to C$ be a surjective morphism of complex projective varieties, where $C$ is a smooth curve. There always exists a \emph{multisection} of $u$, in other words, an algebraic curve $X$, which we can assume to be smooth and projective, together with a branched cover $f : X \to C$, and a morphism $s : X \to Y$, making the diagram
\[
\xymatrix{& Y \ar[d]^u \\ X \ar[ur]^s \ar[r]_f & C}
\]
commute.
\end{multisection}

\begin{proof}
By induction on $\dim Y$. If $Y$ itself is a curve, then letting $X \stackrel{s}{\to} Y$ be its normalization does the trick. Now let $Y \subseteq \mathbb{CP}^N$ have arbitrary dimension, and suppose the result holds for all varieties of lower dimension. We can assume that $Y$ is irreducible. Let $H \subset \mathbb{CP}^N$ be a hyperplane that doesn't contain $Y$. Then the irreducible components of $Y \cap H$ have dimension $\dim Y - 1$ and at least one of them is mapped onto $C$ by the map $u: Y \to C$, hence we can apply the induction hypothesis to the restriction of $u$ to this component.
\end{proof}

Going back to the map $u : Y \to \mathbb{CP}^1$ we have constructed previously and applying this lemma, we get a multisection of $u$ over some branched cover $f : X \to \mathbb{CP}^1$. We proceed as follows in four steps to arrive at a contradiction.

\begin{enumerate}
\item[1.] We take the fibred product $Z$ of $\pi : \Tw(M) \to \mathbb{CP}^1$ and $f : X \to \mathbb{CP}^1$ as in the diagram
\[
\xymatrix{Z \ar[r]^-{\varphi} \ar[d]_\rho & \Tw(M) \ar[d]^\pi \\ X \ar[r]_f & \mathbb{CP}^1}
\]
and use the multisection obtained above to construct a subsheaf $\mathcal{F} \subset \varphi^*E$ of rank $s$ on $Z$. 
\item[2.] We take the pushforward of $\mathcal{F}$ and $\varphi^*E$ along $\varphi$ to obtain a rank $s$ subsheaf $\varphi_*(\mathcal{F}) \subset \varphi_*(\varphi^* E)$ over $\Tw(M)$. We show that $\varphi_*(\varphi^* E)$ is a direct sum of copies of $E$ twisted by some divisors on $\Tw(M)$:
\[
\varphi_*(\varphi^* E) \cong E(D_1) \oplus \ldots \oplus E(D_d).
\]
\item[3.] In view of the above direct sum decomposition of $\varphi_*(\varphi^* E)$ and the irreducibility of $E$, we show that the subsheaf $\varphi_*(\mathcal{F}) \subset E(D_1) \oplus \ldots \oplus E(D_d)$ is essentially isomorphic (in a sense to be made precise) to a direct sum of some of the $E(D_1), \ldots, E(D_d)$, say $E(D_1) \oplus \ldots \oplus E(D_t)$ with $t < d$.
\item[4.] Identifying $\varphi_*(\varphi^* E)$ with $E(D_1) \oplus \ldots \oplus E(D_d)$ and $\varphi_*(\mathcal{F})$ with $E(D_1) \oplus \ldots \oplus E(D_t)$, we show that if $E$ is simple over a fibre of the projection $\pi : \Tw(M) \to \mathbb{CP}^1$, any sheaf monomorphism from $E(D_1) \oplus \ldots \oplus E(D_t)$ to $E(D_1) \oplus \ldots \oplus E(D_d)$ has a particularly simple form, namely that of a $d \times t$ matrix of meromorphic functions from $\mathbb{CP}^1$. From this we can get a contradiction to the fact that $\mathcal{F}$ is a proper subsheaf of $\varphi^* E$ of lower rank on $Z$.
\end{enumerate}

To sum up the above in one sentence, the irreducibility of $E$ and the fact that is simple over a fibre of $\pi$ put rigid conditions on subsheaves of direct sums of copies of $E$ on $\Tw(M)$, from which one can conclude that the pullback bundle $\varphi^* E$ on $Z$ is irreducible as well, and this gives a contradiction to the existence of a subsheaf of $\varphi^* E$ constructed using Lemma \ref{thm:multisect}.

Before proceeding with the rest of the proof of Theorem \ref{thm:result}, we need one more technical result. Recall that so far all our sheaves have been sheaves of $\mathcal{O}$-modules, where $\mathcal{O}$ is the structure sheaf of the corresponding manifold. However, in our argument it will be useful to pass to a different category. Let $\mathcal{M}$ denote the sheaf of meromorphic functions on $\mathbb{CP}^1$, and let $\pi^* \mathcal{M}$ be its pullback to $\Tw(M)$ along the twistor projection $\pi : \Tw(M) \to \mathbb{CP}^1$. $\pi^* \mathcal{M}$ is not coherent, but it is a sheaf of rings on $\Tw(M)$ (in fact, a sheaf of $\mathcal{O}_{\Tw(M)}$-algebras), and so it gives rise to the corresponding abelian category $\pi^*\mathcal{M}\mathrm{\textbf{-Mod}}$. Note that the tensoring functor
\[
\begin{array}{ccccc}
- \otimes \pi^* \mathcal{M} & : & \mathcal{O}_{\Tw(M)}\mathrm{\textbf{-Mod}} & \longrightarrow & \pi^*\mathcal{M}\mathrm{\textbf{-Mod}}, \\
&& \mathcal{G} & \longmapsto & \mathcal{G} \otimes \pi^* \mathcal{M}
\end{array}
\]
is exact. Indeed, this follows from the fact that $\mathcal{M}$ is flat over $\mathbb{CP}^1$, which is a consequence of the fact that a sheaf over a smooth curve is flat if and only if its stalks are torsion-free.

Recall from Proposition \ref{thm:divisors} that taking pullbacks along the map $\pi : \Tw(M) \to \mathbb{CP}^1$ gives a one-to-one correspondence between divisors on $\mathbb{CP}^1$ and those on $\Tw(M)$. Given a divisor $D$ on $\mathbb{CP}^1$, we will denote by the same letter $D$ the corresponding divisor on $\Tw(M)$, and vice versa. The corresponding line bundle on $\mathbb{CP}^1$ will be denoted by $\mathcal{O}_{\mathbb{CP}^1}(D)$, and on $\Tw(M)$ by $\mathcal{O}_{\Tw(M)}(D)$. For a sheaf $\mathcal{F}$ of $\mathcal{O}_{\Tw(M)}$-modules, we denote
\[
\mathcal{F}(D) := \mathcal{F} \otimes_{\mathcal{O}_{\Tw(M)}} \mathcal{O}_{\Tw(M)}(D),
\]
and similarly for sheaves of $\mathcal{O}_{\mathbb{CP}^1}$-modules. Given a $\mathcal{O}_{\Tw(M)}$-sheaf $\mathcal{F}$ on $\Tw(M)$, for any pair of divisors $D \le D'$ we have a natural morphism $\mathcal{F}(D) \to \mathcal{F}(D')$. With these morphisms, the collection $\left\{ \mathcal{F}(D) \right\}$ forms a direct system of sheaves on $\Tw(M)$. For any divisor $D$, tensoring the inclusion $\mathcal{O}_{\Tw(M)}(D) \hookrightarrow \pi^* \mathcal{M}$ with $\mathcal{F}$ gives a morphism $\mathcal{F}(D) \to \mathcal{F} \otimes \pi^*\mathcal{M}$, and the morphisms thus obtained are compatible with the direct system, hence we have an induced map
\begin{equation} \label{dir-lim}
\varinjlim \mathcal{F}(D) \longrightarrow \mathcal{F} \otimes \pi^*\mathcal{M}
\end{equation}
of sheaves of $\mathcal{O}_{\Tw(M)}$-modules. As the following lemma shows, it is an isomorphism.

\theoremstyle{plain}
\newtheorem{direct-lim}[result]{Lemma}
\begin{direct-lim} \label{thm:direct-lim}
For any $\mathcal{O}_{\Tw(M)}$-sheaf $\mathcal{F}$ on $\Tw(M)$, the map \eqref{dir-lim} is an isomorphism. Furthermore, if $\mathcal{F}$ is a torsion-free coherent sheaf, then the maps $\mathcal{F}(D) \to \mathcal{F} \otimes \pi^*\mathcal{M}$ inducing \eqref{dir-lim} are all monomorphisms, and for any coherent sheaf $\mathcal{G}$ of $\mathcal{O}_{\Tw(M)}$-modules, any morphism $\mathcal{G} \to \mathcal{F} \otimes \pi^*\mathcal{M}$ factors through some $\mathcal{F}(D)$ in the category $\mathcal{O}_{\Tw(M)}$\emph{$\textrm{\textbf{-Mod}}$}, as in the diagram
\[
\xymatrix{ & \mathcal{F} \otimes \pi^* \mathcal{M} \\ \mathcal{G} \ar[ur] \ar@{-->}[r] & \mathcal{F}(D) \ar@{^{(}->}[u]}
\]
\end{direct-lim}

\begin{proof}
First observe that, on $\mathbb{CP}^1$, the natural inclusions $\mathcal{O}_{\mathbb{CP}^1}(D) \hookrightarrow \mathcal{M}$ are compatible with the direct system of sheaves $\left\{ \mathcal{O}_{\mathbb{CP}^1}(D) \right\}$ and thus induce a map
\begin{equation} \label{dir-lim-P1}
\varinjlim \mathcal{O}_{\mathbb{CP}^1}(D) \longrightarrow \mathcal{M},
\end{equation}
which is easily seen to be an isomorphism. If $\mathcal{F}$ is any $\mathcal{O}_{\Tw(M)}$-sheaf on $\Tw(M)$, then for any $x \in \Tw(M)$, the stalk map of \eqref{dir-lim} at $x$, which looks like
\begin{equation} \label{dir-lim-stalks}
\varinjlim \mathcal{F}_x \otimes_{\mathcal{O}_{\mathbb{CP}^1, \pi(x)}} \mathcal{O}_{\mathbb{CP}^1, \pi(x)}(D) \longrightarrow \mathcal{F}_x \otimes_{\mathcal{O}_{\mathbb{CP}^1, \pi(x)}} \mathcal{M}_{\pi(x)},
\end{equation}
is just the tensor product of the stalk map of \eqref{dir-lim-P1} at $\pi(x)$ (which is an isomorphism) with $\mathcal{F}_x$. This shows that \eqref{dir-lim} is an isomorphism at $x$ for arbitrary $x \in \Tw(M)$.
 
If now $\mathcal{F}$ is a torsion-free sheaf on $\Tw(M)$, then for any $x \in \Tw(M)$, $\mathcal{F}_x$ is a flat $\mathcal{O}_{\mathbb{CP}^1,\pi(x)}$-module, hence tensoring the inclusion $\mathcal{O}_{\mathbb{CP}^1, \pi(x)}(D) \hookrightarrow \mathcal{M}_{\pi(x)}$ with $\mathcal{F}_x$ produces an injective map
\[
\mathcal{F}_x \otimes_{\mathcal{O}_{\mathbb{CP}^1, \pi(x)}} \mathcal{O}_{\mathbb{CP}^1, \pi(x)}(D) \longhookrightarrow \mathcal{F}_x \otimes_{\mathcal{O}_{\mathbb{CP}^1, \pi(x)}} \mathcal{M}_{\pi(x)}.
\]
This shows that $\mathcal{F}(D) \to \mathcal{F} \otimes \pi^*\mathcal{M}$ is a monomorphism for any $D$. Now suppose that $\mathcal{G}$ is a coherent $\mathcal{O}_{\Tw(M)}$-sheaf, and let $\mathcal{G} \to \mathcal{F} \otimes \pi^*\mathcal{M}$ be any morphism. Since $\mathcal{G}$ is in particular of finite type, for every point $x \in \Tw(M)$ there exists a neighborhood $U \ni x$ and sections $s_1, \ldots, s_n \in \mathcal{G}(U)$ that generate $\mathcal{G}$ over $U$. Using the isomorphism \eqref{dir-lim-stalks}, we can conclude that there exists a divisor $D$ such that the image of each $s_i$ under the stalk map $\mathcal{G}_x \to \left(\mathcal{F} \otimes \pi^*\mathcal{M}\right)_x$ lies in $\mathcal{F}(D)_x \subset \left(\mathcal{F} \otimes \pi^*\mathcal{M}\right)_x$. By shrinking $U$ if necessary, we can conclude that the map $\mathcal{G} \to \mathcal{F} \otimes \pi^*\mathcal{M}$ factors through $\mathcal{F}(D)$ on $U$:
\[
\xymatrix{ & \left.\left(\mathcal{F} \otimes \pi^* \mathcal{M}\right)\right|_U \\ \left.\mathcal{G}\right|_U \ar[ur] \ar@{-->}[r] & \left.\mathcal{F}(D)\right|_U \ar@{^{(}->}[u]}
\]
Repeating the same argument for every point of $\Tw(M)$ and using compactness, we can conclude that there exist open sets $U_1, \ldots U_k$ that cover $\Tw(M)$ and divisors $D_1, \ldots, D_k$ such that $\mathcal{G} \to \mathcal{F} \otimes \pi^*\mathcal{M}$ factors through $\mathcal{F}(D_j)$ on $U_j$, similarly to the above. Then choosing $D$ such that $D_j \le D \, \forall j$, it's not hard to see that $\mathcal{G} \to \mathcal{F} \otimes \pi^*\mathcal{M}$ factors through $\mathcal{F}(D)$ on the whole $\Tw(M)$.
\end{proof}

\begin{proof}[Proof of converse of Theorem \ref{thm:result} in the case that $E$ is simple over some fibre of $\pi$]

Let \newline $E$ be an irreducible vector bundle of rank $r$ on $\Tw(M)$ such that its restriction to some fibre of $\pi : \Tw(M) \to \mathbb{CP}^1$ is a simple vector bundle. Arguing by contradiction, we assume that $E_I$ is non-stable for infinitely many $I \in \mathbb{CP}^1$.

\subsection*{Step 1} By Theorem \ref{thm:stable-twistor}, $E$ is non-stable for all $I \in \mathbb{CP}^1$, and by the discussion preceding Lemma \ref{thm:multisect}, there exists a number $1 \le s \le r-1$ and a line bundle $L$ on $\Tw(M)$ such that for all $I \in \mathbb{CP}^1$, there are non-trivial morphisms
\[
L_I \longrightarrow C_s(E_I) \subseteq \Lambda^s E_I
\]
over $M_I = \pi^{-1}(I)$, where as usual $L_I$ is the restriction of $L$ to $M_I$ and $C_s(E_I)$ denotes the cone of exterior monomials in $\Lambda^s E_I$. If $s = 1$ or $r-1$, $C_s(E_I) = \Lambda^s E_I$, and an argument entirely analogous to the case $\rk E = 3$ shows that these morphisms can be assembled into a line subsheaf of $\Lambda^sE$ on $\Tw(M)$ taking values in $C_s(E)$, which in turn contradicts the irreducibility of $E$, proving that $E$ is generically fibrewise stable. Assume $1 < s < r-1$.

As explained previously, $\pi_*(L^* \otimes \Lambda^sE)$ is a nonzero vector bundle on $\mathbb{CP}^1$ whose generic fibre is isomorphic to $\Hom_{M_I}(L_I, \Lambda^s E_I)$. Taking the corresponding projective bundle $\mathbb{P}\left(\pi_*(L^* \otimes \Lambda^sE)\right)$ over $\mathbb{CP}^1$, we have the closed algebraic subvariety
\[
\xymatrix{Y := \left\{\textrm{Classes of maps } L_I \hookrightarrow C_s(E_I) \subseteq \Lambda^s E_I \right\} \ar@{^{(}->}[r] \ar[dr]_{u} & \mathbb{P}(\pi_*(L^* \otimes \Lambda^s E)) \ar[d]^v \\ & \mathbb{CP}^1,}
\]
where $u$ is surjective. Applying Lemma \ref{thm:multisect}, there exists a multisection
\[
\xymatrix{& Y \ar[d]^u \\ X \ar[ur]^s \ar[r]_f & \mathbb{CP}^1}
\]
where $f : X \to \mathbb{CP}^1$ is branched cover of degree $d$. Taking the fibred product
\[
Y \times_{\mathbb{CP}^1} X \subseteq \mathbb{P}\left(f^*(\pi_*(L^* \otimes \Lambda^s E))\right),
\]
our multisection $s : X \to Y$ gives a section of the morphism $Y \times_{\mathbb{CP}^1} X \to X$, so we obtain a line subbundle
\[
0 \longrightarrow K \longrightarrow f^*(\pi_*(L^* \otimes \Lambda^s E))
\]
on $X$. Let $Z$ denote the fibred product of $f : X \to \mathbb{CP}^1$ and the twistor projection $\pi : \Tw(M) \to \mathbb{CP}^1$, as in the diagram
\begin{equation} \label{diagramma}
\xymatrix{Z \ar[r]^-{\varphi} \ar[d]_\rho & \Tw(M) \ar[d]^\pi \\ X \ar[r]_f & \mathbb{CP}^1}
\end{equation}
On $X$, there is a canonical morphism,
\[
f^*(\pi_*(L^* \otimes \Lambda^s E)) \longrightarrow \rho_*(\varphi^*(L^* \otimes \Lambda^s E)),
\]
which over a generic point is easily seen to be an isomorphism, since $f : X \to \mathbb{CP}^1$ is a finite map. Composing this with the line subbundle $K \to f^*(\pi_*(L^* \otimes \Lambda^s E))$ constructed above, we get a line subsheaf
\[
0 \longrightarrow K \longrightarrow \rho_*(\varphi^*(L^* \otimes \Lambda^s E)) = \rho_*(\left[\varphi^*L\right]^* \otimes \varphi^*\Lambda^s E)
\]
on $X$, and pulling back along $\rho$ to $Z$, we take the composition
\[
\rho^*(K) \longrightarrow \rho^*\rho_*(\left[\varphi^*L\right]^* \otimes \varphi^*\Lambda^s E) \longrightarrow \left[\varphi^*L\right]^* \otimes \varphi^*\Lambda^s E
\]
with the corresponding evaluation map. This composition is a monomorphism since it is nonzero, and it takes values in $\left[\varphi^*L\right]^* \otimes \varphi^*C_s(E)$ by construction. Tensoring with $\varphi^*L$, we get a line subsheaf of $\varphi^*\Lambda^s E = \Lambda^s (\varphi^* E)$ which takes values in $\varphi^*C_s(E) = C_s(\varphi^*E)$. By the correspondence \eqref{correspondence}, this line subsheaf gives rise to a rank $s$ normal subsheaf $\mathcal{F} \subset \varphi^* E$ on $Z$.

\subsection*{Step 2} Taking the pushforward of the sheaf monomorphism $\mathcal{F} \hookrightarrow \varphi^* E$ that we just obtained along the map $\varphi$, we obtain by the left-exactness of $\varphi_*$ a sheaf monomorphism
\[
\gamma : \varphi_*(\mathcal{F}) \hooklongrightarrow \varphi_*(\varphi^* E)
\]
on $\Tw(M)$. Since $\mathcal{F}$ was normal over $Z$, $\varphi_*(\mathcal{F})$ is normal over $\Tw(M)$. As for $\varphi_*(\varphi^* E)$, it happens to be a vector bundle whose structure can be described nicely in terms of the original bundle $E$. In the diagram \eqref{diagramma} we have an isomorphism
\[
\pi^*(f_*(\mathcal{O}_X)) \cong \varphi_*(\rho^*(\mathcal{O}_X))
\]
(see Theorem III.3.10 and its corollaries in \cite{banica}). Using the Birkhoff-Grothendieck theorem, we can write
\[
\varphi_*(\mathcal{O}_Z) = \varphi_*(\rho^*(\mathcal{O}_X)) \cong \pi^*(f_*(\mathcal{O}_X)) \cong \pi^*\left(\bigoplus_{l=1}^d \mathcal{O}_{\mathbb{CP}^1}(D_l)\right) = \bigoplus_{l=1}^d \mathcal{O}_{\Tw(M)}(D_l),
\]
where $D_1, \ldots, D_d$ are some divisors on $\mathbb{CP}^1$. Using this decomposition and the projection formula, we have
\[
\varphi_*(\varphi^* E) = \varphi_*(\varphi^*E \otimes \mathcal{O}_Z) \cong E \otimes \varphi_*(\mathcal{O}_Z) \cong E(D_1) \oplus \ldots \oplus E(D_d).
\]
We will denote $E_l := E(D_l)$ for $1 \le l \le d$. With this identification, the sheaf monomorphism $\gamma$ has the form
\[
\gamma : \varphi_*(\mathcal{F}) \hooklongrightarrow \varphi_*(\varphi^* E) \cong E_1 \oplus \ldots \oplus E_d.
\]

\subsection*{Step 3}
We would now like to show that there exists a choice of a subset $\left\{i_1, \ldots, i_t \right\} \subseteq \left\{1, \ldots, d\right\}$ such that the composition
\[
\varphi_*(\mathcal{F}) \stackrel{\gamma}{\longrightarrow} E_1 \oplus \ldots \oplus E_d \longrightarrow E_{i_1} \oplus \ldots \oplus E_{i_t}
\]
(where the second arrow is the usual projection) is a monomorphism of sheaves with quotient being a torsion sheaf. Let $1 \le j \le d$ and look at the composition
\[
\varphi_*(\mathcal{F}) \stackrel{\gamma}{\longrightarrow} E_1 \oplus \ldots \oplus E_d \longrightarrow \bigoplus_{l \ne j} E_l = E_1 \oplus \ldots \oplus \widehat{E_j} \oplus \ldots \oplus E_d.
\]
Let $K_j$ denote the kernel of this composition. We have the following diagram with exact rows
\[
\xymatrix{0 \ar[r] & K_j \ar[r] \ar@{-->}[d] & \varphi_*(\mathcal{F}) \ar[r] \ar@{^{(}->}[d]^{\gamma} & \bigoplus_{l \ne j} E_l  \ar@{=}[d] & \\ 0 \ar[r] & E_j \ar[r] & E_1 \oplus \ldots \oplus E_d \ar[r] & \bigoplus_{l \ne j} E_l \ar[r] &0}
\]
It follows from the irreducibility of $E_j = E(D_j)$ that the the induced morphism $K_j \to E_j$ is either zero or generically an isomorphism. If the latter is true for every $j$ from 1 to $d$, then $\rk \varphi_*(\mathcal{F}) = \rk \left[E_1 \oplus \ldots \oplus E_d\right]$, but this cannot be as $\varphi_*(\mathcal{F})$ is a subsheaf of $\varphi_*(\varphi^* E)$ of lower rank. Fixing an index $j$ such that $K_j = 0$, the composition
\[
\varphi_*(\mathcal{F}) \stackrel{\gamma}{\longrightarrow} E_1 \oplus \ldots \oplus E_d \longrightarrow \bigoplus_{l \ne j} E_l
\]
must be a monomorphism. If $\rk \varphi_*(\mathcal{F}) = \rk \bigoplus_{l \ne j} E_l$, we stop here. If not, we repeat the argument above with $\left\{1, \ldots, d\right\}$ replaced by $\left\{1, \ldots, \widehat{j}, \ldots, d\right\}$ to conclude the existence of an index $k \in \left\{1, \ldots, \widehat{j}, \ldots, d\right\}$ such that the composition
\[
\varphi_*(\mathcal{F}) \longrightarrow \bigoplus_{l \ne j} E_l \longrightarrow \bigoplus_{l \ne j, k} E_l
\]
is still a monomorphism. Continuing in this manner, we eventually arrive at a monomorphism $\varphi_*(\mathcal{F}) \hookrightarrow E_{i_1} \oplus \ldots \oplus E_{i_t}$ with $\rk \varphi_*(\mathcal{F}) = \rk \left[E_{i_1} \oplus \ldots \oplus E_{i_t}\right]$, whose quotient is clearly a torsion sheaf. We cannot have $t = 0$ as this would imply $\varphi_*(\mathcal{F}) = 0$, contrary to the construction of $\mathcal{F}$. Rearranging indices if necessary, we can assume that $E_{i_1} \oplus \ldots \oplus E_{i_t} = E_1 \oplus \ldots \oplus E_t$.

We denote the sheaf monomorphism that we just constructed by
\[
\mu : \varphi_*(\mathcal{F}) \longhookrightarrow E_1 \oplus \ldots \oplus E_t.
\]
In the short exact sequence
\begin{equation} \label{***}
0 \longrightarrow \varphi_*(\mathcal{F}) \stackrel{\mu}{\longrightarrow} E_1 \oplus \ldots \oplus E_t \longrightarrow \mathcal{T} \longrightarrow 0,
\end{equation}
$\mathcal{T}$ is a torsion sheaf, as noted above. Thus, the morphism $\mu : \varphi_*(\mathcal{F}) \to E_1 \oplus \ldots \oplus E_t$ has an inverse outside a subset of positive codimension in $\Tw(M)$, and we would like to extend this inverse to all of $\Tw(M)$ in some way.

We start by observing that $\Supp(\mathcal{T})$ has pure codimension 1 in $\Tw(M)$. Indeed, for any open neighborhood $U \subseteq \Tw(M)$ and $A \subset U$ a subset of codimension $\ge 2$, we have the following diagram with exact rows:
\[
\xymatrix{ 0 \ar[r] & \varphi_*(\mathcal{F})(U) \ar[r]^-{\mu(U)} \ar[d]^\cong & \left[E_1 \oplus \ldots \oplus E_t\right](U) \ar[r] \ar[d]^\cong & \mathcal{T}(U) \ar[d] \\ 0 \ar[r] & *+<1pc>{\varphi_*(\mathcal{F})(U \setminus A)} \ar[r]^-{\mu(U \setminus A)} & *+<1pc>{\left[E_1 \oplus \ldots \oplus E_t\right](U \setminus A)} \ar[r] & \mathcal{T}(U \setminus A) }
\]
The first and second vertical arrows are isomorphisms because the corresponding sheaves are normal. If $\zeta \in \mathcal{T}(U)$ is some section such that $\Supp(\zeta) = A$, then, by shrinking $U$ if necessary, we can assume that it comes from some section $\xi \in \left[E_1 \oplus \ldots \oplus E_t\right](U)$. As the restriction of $\zeta$ to $U \setminus A$ is zero, it follows by the exactness of the second row that the restriction of $\xi$ to $U \setminus A$ comes from $\varphi_*(\mathcal{F})(U \setminus A)$. But then the same must hold over $U$, so by the exactness of the first row, $\zeta$ must be zero over $U$, which is a contradiction. This means that $\Supp(\mathcal{T})$ lies on a divisor.

We now pass to the category $\pi^*\mathcal{M}\mathrm{\textbf{-Mod}}$. As we have noted previously, the tensoring functor $- \otimes \pi^* \mathcal{M}$ is exact, so applying it to \eqref{***}, we get the short exact sequence
\[
0 \longrightarrow \varphi_*(\mathcal{F}) \otimes \pi^* \mathcal{M} \xrightarrow{\mu \otimes \pi^* \mathcal{M}} \left[E_1 \oplus \ldots \oplus E_t\right] \otimes \pi^* \mathcal{M} \longrightarrow \mathcal{T} \otimes \pi^* \mathcal{M} \longrightarrow 0
\]
in $\pi^*\mathcal{M}\mathrm{\textbf{-Mod}}$. By the discussion in the previous paragraph, we know that the support of every nonzero element $\zeta \in \mathcal{T}_x$, for any $x \in \Tw(M)$, lies on the hypersurface $V = \pi^{-1}(\pi(x))$, hence $\zeta$ is annihilated by a sufficiently high power of a local coordinate about $\pi(x)$ in $\mathbb{CP}^1$. Thus, $\mathcal{T} \otimes \pi^* \mathcal{M} = 0$, and so $\mu \otimes \pi^* \mathcal{M} : \varphi_*(\mathcal{F}) \otimes \pi^* \mathcal{M} \to \left[E_1 \oplus \ldots \oplus E_t\right] \otimes \pi^* \mathcal{M}$ is an isomorphism of $\pi^* \mathcal{M}$-modules. We denote its inverse in the category $\pi^*\mathcal{M}\mathrm{\textbf{-Mod}}$ by
\[
\eta : \left[E_1 \oplus \ldots \oplus E_t\right] \otimes \pi^* \mathcal{M} \longrightarrow \varphi_*(\mathcal{F}) \otimes \pi^* \mathcal{M}.
\]
Applying Lemma \ref{thm:direct-lim}, there exists a divisor $D$ and a morphism $E_1 \oplus \ldots \oplus E_t \to \varphi_*(\mathcal{F})$ (which we also denote by $\eta$) which completes the following diagram in the category $\mathcal{O}_{\Tw(M)}\mathrm{\textbf{-Mod}}$:
\[
\xymatrix{\left[E_1 \oplus \ldots \oplus E_t\right] \otimes \pi^* \mathcal{M} \ar[r]^-\eta & \varphi_*(\mathcal{F}) \otimes \pi^* \mathcal{M} \\ *+<1pc>{E_1 \oplus \ldots \oplus E_t} \ar@{^{(}->}[u] \ar@{-->}[r]^\eta & *+<1pc>{\varphi_*(\mathcal{F})(D)} \ar@{^{(}->}[u]}
\]
The morphism $\eta : E_1 \oplus \ldots \oplus E_t \to \varphi_*(\mathcal{F})(D)$ is a partial inverse to the morphism $\mu : \varphi_*(\mathcal{F}) \to E_1 \oplus \ldots \oplus E_t$ constructed above, in the sense of the following commutative diagram:
\begin{equation} \label{inverses}
\xymatrix{\varphi_*(\mathcal{F})(D) \ar[r]^-{\mu(D)} & \left[E_1 \oplus \ldots \oplus E_t\right](D) \\ *+<1pc>{\varphi_*(\mathcal{F})} \ar[r]^-{\mu} \ar@{^{(}->}[u] & *+<1pc>{E_1 \oplus \ldots \oplus E_t} \ar@{^{(}->}[u] \ar[ul]_{\eta}}
\end{equation}
Note that outside $\Supp(D)$, $\mu$ and $\eta$ are bona fide inverses of each other. 

\subsection*{Step 4} In Step 2 we have constructed the sheaf monomorphism $\gamma : \varphi_*(\mathcal{F}) \hookrightarrow \varphi_*(\varphi^* E) \cong E_1 \oplus \ldots \oplus E_d$ on $\Tw(M)$ as the pushforward by the map $\varphi$ from the diagram \eqref{diagramma} of the sheaf inclusion $\mathcal{F} \hookrightarrow \varphi^*E$ on $Z$. Identifying $\varphi_*(\mathcal{F})$ with $E_1 \oplus \ldots \oplus E_t$ via the diagram \eqref{inverses} above, we will show that the morphism $\gamma$ has a particularly simple form. Look at the composition
\begin{equation} \label{gamma-matrix}
E_1(-D) \oplus \ldots \oplus E_t(-D) \xrightarrow{\eta(-D)} \varphi_*(\mathcal{F}) \stackrel{\gamma}{\longrightarrow} E_1 \oplus \ldots \oplus E_d.
\end{equation}
Take any summand $E_j(-D)$, $1 \le j \le t$, on the left, and any summand $E_i$, $1 \le i \le d$, on the right. The morphisms from $E_j(-D)$ to $E_i$ on $\Tw(M)$ are given by:
\[
\Hom\left(E_j(-D), E_i\right) = \Hom\left(E(D_j - D), E(D_i)\right) =
\]
\[
= H^0(\Tw(M), E^*(-D_j + D) \otimes E(D_i)) = H^0(\Tw(M), (E^* \otimes E)(-D_j + D + D_i)) =
\]
\[
= H^0(\mathbb{CP}^1, \pi_*\left[(E^* \otimes E)(-D_j + D + D_i)\right]) = H^0(\mathbb{CP}^1, \pi_*(E^* \otimes E)(-D_j + D + D_i)),
\]
where we have used the projection formula in the last line with respect to the twistor projection $\pi : \Tw(M) \to \mathbb{CP}^1$. The pushforward sheaf $\pi_*(E^* \otimes E)$ on $\mathbb{CP}^1$ is locally free since it's torsion-free, so it splits as a direct sum of line bundles. Since the vector bundle $E$ on $\Tw(M)$ is irreducible, it is stable, hence it is simple (see \cite{lubke-teleman}, Proposition 1.4.5), so we have $H^0(\mathbb{CP}^1, \pi_*(E^* \otimes E)) = H^0(\Tw(M), E^* \otimes E) = \mathbb{C}$. It follows that in the direct sum decomposition of $\pi_*(E^* \otimes E)$, there is exactly one summand of the form $\mathcal{O}_{\mathbb{CP}^1}$, while all other summands (if any) are negative line bundles.

On the other hand, it's clear that for any $I \in \mathbb{CP}^1$, $\dim H^0(M_I, E_I^* \otimes E_I) \ge 1$, and since the restriction of $E$ to some fibre of $\pi$ is simple, we know that equality is attained for at least one $I$. Applying the semicontinuity theorem (\cite{okonek}, p. 5) to $E^* \otimes E$, we conclude that $H^0(M_I, E_I^* \otimes E_I)$ has rank 1 for $I$ in a nonempty Zariski open subset of $\mathbb{CP}^1$, and so, by the previous paragraph, $\pi_*(E^* \otimes E) = \mathcal{O}_{\mathbb{CP}^1}$. With this in mind, we have
\[
\Hom\left(E_j(-D), E_i\right) = H^0(\mathbb{CP}^1, \mathcal{O}_{\mathbb{CP}^1}(-D_j + D + D_i)),
\]
and it follows from this that the composition \eqref{gamma-matrix} is simply a $d \times t$ matrix of meromorphic functions from $\mathbb{CP}^1$ that we will denote by $A$.

We will now complete our argument by showing that the above description of the morphism $\gamma : \varphi_*(\mathcal{F}) \to \varphi_*(\varphi^* E)$ as a matrix morphism is incompatible with the fact that $\mathcal{F}$ is a proper subsheaf of $\varphi^* E$ of lower rank on $Z$. Recall that in the diagram \eqref{diagramma} $f$ and $\varphi$ are branched coverings of degree $d$. We can choose a nonempty open neighborhood $U \subseteq \Tw(M)$ such that

\begin{enumerate}
\item[(i)] $U$ is evenly covered by $\varphi$, that is, $\varphi^{-1}(U) \subseteq Z$ is a disjoint union of open neighborhoods $U_1, \ldots, U_d$ such that $\forall 1 \le i \le d$, $\varphi$ restricted to $U_i$ is an isomorphism
\[
\varphi|_{U_i} : U_i \stackrel{\cong}{\longrightarrow} U;
\]

\item[(ii)] $U$ does not intersect $\Supp(D)$ in $\Tw(M)$;

\item[(iii)] the line bundles $\mathcal{O}_{\Tw(M)}(D_1), \ldots, \mathcal{O}_{\Tw(M)}(D_d)$ trivialize on the open set $U$.
\end{enumerate}

We now look at the restriction of the sheaf monomorphism $\gamma : \varphi_*(\mathcal{F}) \to \varphi_*(\varphi^* E)$ to $U \subseteq \Tw(M)$. $\gamma$ can be described in two different ways. Firstly, as $U$ is evenly covered by $U_1, \ldots, U_d$, we have
\[
\left.\varphi_*(\mathcal{F})\right|_U \cong \left.\mathcal{F}\right|_{U_1} \oplus \ldots \oplus \left.\mathcal{F}\right|_{U_d}, \quad \left.\varphi_*(\varphi^* E)\right|_U \cong \left.\varphi^* E\right|_{U_1} \oplus \ldots \oplus \left.\varphi^* E\right|_{U_d},
\]
and $\gamma$ is simply the direct sum of the monomorphisms $\left.\mathcal{F}\right|_{U_i} \hookrightarrow \left.\varphi^* E\right|_{U_i}$ defining $\mathcal{F}$ as a subsheaf of $\varphi^* E$ on $Z$:
\[
\left.\gamma\right|_U : \left.\mathcal{F}\right|_{U_1} \oplus \ldots \oplus \left.\mathcal{F}\right|_{U_d} \longrightarrow \left.\varphi^* E\right|_{U_1} \oplus \ldots \oplus \left.\varphi^* E\right|_{U_d}.
\]
Secondly, outside $\Supp(D)$ (and in particular on $U$), the morphsim $\eta(-D)$ in the composition \eqref{gamma-matrix} is an isomorphism, hence we can identify $\varphi_*(\mathcal{F})$ with $E_1(-D) \oplus \ldots \oplus E_t(-D)$, while $\varphi_*(\varphi^* E)$ can be identified with $E_1 \oplus \ldots \oplus E_d$ in the usual way:
\[
\left.\varphi_*(\mathcal{F})\right|_U \cong \left.E_1(-D)\right|_{U} \oplus \ldots \oplus \left.E_t(-D)\right|_{U}, \quad \left.\varphi_*(\varphi^* E)\right|_U \cong \left.E_1\right|_{U} \oplus \ldots \oplus \left.E_d\right|_{U}.
\]
With these identifications, $\left.\gamma\right|_U$ is simply the matrix morphism $A$ above:
\[
A : \left.E_1(-D)\right|_{U} \oplus \ldots \oplus \left.E_t(-D)\right|_{U} \longrightarrow \left.E_1\right|_{U} \oplus \ldots \oplus \left.E_d\right|_{U}.
\]
Note that, in view of \eqref{inverses}, the composition of $A$ with the projection onto the first $t$ summands of $\left.E_1\right|_{U} \oplus \ldots \oplus \left.E_d\right|_{U}$ is an isomorphism on $U$.

To reconcile these two descriptions of the morphism $\gamma$, observe that by choice of $U \subseteq \Tw(M)$, we have two trivializations of the vector bundle $\varphi_*(\mathcal{O}_Z)$ on $U$:
\[
\left.\varphi_*(\mathcal{O}_Z)\right|_U \stackrel{\cong}{\longrightarrow} \mathcal{O}_{U_1} \oplus \ldots \oplus \mathcal{O}_{U_d} \cong \mathcal{O}_U^{\oplus d}, \ \left.\varphi_*(\mathcal{O}_Z)\right|_U \stackrel{\cong}{\longrightarrow} \mathcal{O}_{U}(D_1) \oplus \ldots \oplus \mathcal{O}_{U}(D_d) \cong \mathcal{O}_U^{\oplus d}.
\]
Here the first trivialization comes from the fact that $U \subseteq \Tw(M)$ is evenly covered by $U_1, \ldots, U_d \subseteq Z$, while the second trivialization comes from the isomorphism
\[
\varphi_*(\mathcal{O}_Z) \cong \mathcal{O}_{\Tw(M)}(D_1) \oplus \ldots \oplus \mathcal{O}_{\Tw(M)}(D_d)
\]
that was obtained in Step 2. If we denote by $B : \mathcal{O}_U^{\oplus d} \longrightarrow \mathcal{O}_U^{\oplus d}$ the transition function from the second to the first trivialization, the two descriptions of the morphism $\gamma$ on $U$ from the previous paragraph now fit into the following commutative diagram:
\[
\xymatrix@C+3.5pc{ \left.\mathcal{F}\right|_{U_1} \oplus \ldots \oplus \left.\mathcal{F}\right|_{U_d} \ar[r]^-{\left.\gamma\right|_U} & **[l] \left.\varphi^* E\right|_{U_1} \oplus \ldots \oplus \left.\varphi^* E\right|_{U_d} \cong \left.E\right|_U^{\oplus d}
\\
\left.E_1(-D)\right|_{U} \oplus \ldots \oplus \left.E_t(-D)\right|_{U} \ar[r]^-{A} \ar[u]^{\eta(-D)}_\cong & **[l]\left.E_1\right|_{U} \oplus \ldots \oplus \left.E_d\right|_{U} \cong \left.E\right|_U^{\oplus d} \ar@<1ex>[u]^{B}_\cong}
\]
Since $\mathcal{F}$ was constructed as a proper subsheaf of $\varphi^* E$ of lower rank, composing the upper horizontal arrow with the projection onto any direct summand of $\left.E\right|_U^{\oplus d}$ will yield a map which is not surjective at any point of $U$. On the other hand, as noted at the end of the previous paragraph, the composition of the lower horizontal arrow with the projection onto any of the first $t$ direct summands of $\left.E\right|_U^{\oplus d}$ is surjective. These two statements contradict each other, since the right-hand vertical arrow $B : \left.E\right|_U^{\oplus d} \to \left.E\right|_U^{\oplus d}$ is simply a $d \times d$ matrix of holomorphic functions on $U$, everywhere non-singular. We conclude that our original assumption that $E$ is not generically fibrewise stable is wrong, finishing the proof of the theorem.

\end{proof}

\end{document}